\documentclass[11pt]{amsart}

\usepackage{amsmath,amsfonts,amsthm,amscd,amssymb,graphicx}
\numberwithin{equation}{section}

\usepackage{fullpage}
\usepackage{hyperref}

\usepackage{color}

\newtheorem{theorem}{Theorem}[section]

\newtheorem{lemma}[theorem]{Lemma}
\newtheorem{lem}[theorem]{Lemma}
\newtheorem{proposition}[theorem]{Proposition}
\newtheorem{prop}[theorem]{Proposition} 
\newtheorem{corollary}[theorem]{Corollary}

\newtheorem{remark}[theorem]{Remark}
\newtheorem{definition}[theorem]{Definition}

\newcommand{\T}{{\mathbb T}}
\newcommand{\N}{{\mathbb N}}
\newcommand{\Z}{{\mathbb Z}}

\newcommand{\R}{{\mathbb R}}

\newcommand{\pa}{{\partial}}
\newcommand{\na}{{\nabla}}

\newcommand{\eps}{{\varepsilon}}

\newcommand{\Nc}{\mathcal{N}}

\newcommand{\Rc}{\mathcal{R}}

\newcommand{\A}{\widetilde{A}}

\newcommand{\e}{\varepsilon}

\renewcommand{\P}{\mathbb{P}}

\newcommand{\cS}{\mathcal{S}}
\newcommand{\cE}{\mathcal{E}}

\begin{document}


\title[Long times estimates in the non-relativistic limit]{Long time estimates for the Vlasov-Maxwell system in the non-relativistic limit}

\author{Daniel Han-Kwan}
  \address{CMLS \\ \'Ecole polytechnique, CNRS, Universit\'e Paris-Saclay, 91128 Palaiseau Cedex, France}
  \email{daniel.han-kwan@polytechnique.edu}

 \author{Toan T. Nguyen} 
   \address{Department of Mathematics \\ Pennsylvania State University, State College, PA 16802, USA}
  \email{nguyen@math.psu.edu}
 
 \author{Fr\'ed\'eric Rousset}
  \address{Laboratoire de Math\'ematiques d'Orsay (UMR 8628) \\ Universit\'e Paris-Sud et Institut Universitaire de France, 91405 Orsay Cedex, France}
  \email{frederic.rousset@math.u-psud.fr}

\date{\today}

\maketitle

\renewcommand{\thefootnote}{\fnsymbol{footnote}}


\begin{abstract}
In this paper, we study  the Vlasov-Maxwell system  in the non-relativistic limit, that is in the regime where the speed of light is a very large parameter.  We consider data lying in the vicinity of homogeneous equilibria that are stable in the sense of Penrose (for the Vlasov-Poisson system), and prove Sobolev stability estimates that are valid for times which are polynomial in terms of the speed of light and of the inverse of size of initial perturbations.
We build a kind of higher-order Vlasov-Darwin approximation which allows us to reach arbitrarily large powers of the speed of light.


\end{abstract}

\tableofcontents

\section{Introduction}

We study the relativistic Vlasov-Maxwell system
\begin{equation}\label{clVM0} 
\left \{ \begin{aligned}
\partial_t f + \hat{v} \cdot \nabla_x f + (E + \frac{1}{c} \hat{v} \times B)\cdot \nabla_v f & =0,
\\
\frac{1}{c} \partial_t B + \nabla_x \times E = 0, \qquad \nabla_x \cdot E  &= \int_{\R^3} f \, dv  - 1,
\\
- \frac{1}{c} \partial_t E + \nabla_x \times B = \frac{1}{c} \int_{\R^3} \hat{v} f   \, dv  , \qquad \nabla_x \cdot B & =0, \end{aligned}
\right.
\end{equation}
describing the evolution of an electron distribution function $f(t,x,v)$ at time $t\ge 0$, position $x\in \mathbb{T}^3:= \mathbb{R}^3/ \mathbb{Z}^3$, momentum $v \in \mathbb{R}^3$ and relativistic velocity
$$
\hat{v} = \frac{v}{\sqrt{1+  \frac{|v|^2}{c^2}}}. 
$$
Here, $\mathbb{T}^3$ is equipped with the Lebesgue measure which is normalized so that $\mathrm{Leb}(\mathbb{T}^3)=1$. The three-dimensional vector fields $E(t,x), B(t,x)$ are respectively the electric and magnetic fields.
The background ions are assumed to be homogeneous with a constant charge density equal to one. We endow the system with initial conditions $(f_{\vert_{t=0}}, E_{\vert_{t=0}}, B_{\vert_{t=0}} )$ satisfying the compatibility conditions
$$
\na_x \cdot E_{\vert_{t=0}} = \int_{\R^3} f_{\vert_{t=0}} \, dv -1 , \quad \na_x \cdot B_{\vert_{t=0}} = 0.
$$
   
      \begin{remark}{\em 
   We recall that the existence of global smooth solutions to the Vlasov-Maxwell system in three dimensions is at the time of writing still an open problem: see \cite{GS86}, \cite{KS}, \cite{BGP}, and \cite{P}, \cite{LS} (and references therein) for recent advances for the equations set on the whole space $\R^3$. However, global existence is known for the case of lower dimensions, when the space domain is  $\R^2$, see \cite{GSC97, GSC98}, as well as for small data (i.e., close to $0$) when the space domain is  $\R^3$, see \cite{GS-CMP}.
   }
  
   \end{remark}
   
In the relativistic Vlasov-Maxwell system~\eqref{clVM0}, the parameter $c$ is the speed of light; we focus in this work on the regime where $c \to +\infty$, that is known as the \emph{non-relativistic limit of the Vlasov-Maxwell system}. The formal limit is the following classical Vlasov-Poisson system 
\begin{equation}\label{VP} 
\left \{ \begin{aligned}
\partial_t f + {v} \cdot \nabla_x f + E \cdot \nabla_v f & =0,
\\
 \nabla_x \times E = 0, \qquad \nabla_x \cdot E  &=\int_{\R^3} f \, dv - 1 .
\end{aligned}
\right.
\end{equation} 
This formal limit was justified on finite intervals of time in the independent and simultaneous works of Asano-Ukai \cite{AU}, Degond \cite{DEG}, and Schaeffer \cite{SCH}.
In the recent work \cite{HKN}, it was proved that in the {non-relativistic limit}, instabilities may show up in times of order $\log c$, due to instabilities of the underlying Vlasov-Poisson system.

In this work we continue the investigation of large time behaviour of the solutions in the {non-relativistic regime}.
More precisely, we shall study the case of data lying in the neighborhood of \emph{stable} homogeneous equilibria, proving, in sharp contrast with the unstable case, that the approximation by the equilibrium is valid in times which are polynomial in $c$ and with respect to the inverse of the size of the initial perturbation.
 
 For convenience we shall set throughout this text 
$$
\varepsilon = \frac{1}{c},
$$
which has to be seen as a small parameter.
We therefore study the stability properties of  the Vlasov-Maxwell system in the regime of small $\eps$ 
\begin{equation}\label{clVM1} 
\left \{ \begin{aligned}
\partial_t f^\varepsilon + \hat{v} \cdot \nabla_x f^\varepsilon + (E^\varepsilon + \varepsilon \hat{v} \times B^\varepsilon)\cdot \nabla_v f^\varepsilon & =0,
\\
\varepsilon \partial_t B^\varepsilon + \nabla_x \times E^\varepsilon = 0, \qquad \nabla_x \cdot E^\varepsilon  &= \rho(f^\eps) - 1,
\\
- \varepsilon \partial_t E^\varepsilon + \nabla_x \times B^\varepsilon = \varepsilon j(f^\eps), \qquad \nabla_x \cdot B^\varepsilon & =0, 
\end{aligned}
\right.
\end{equation}
in which $\hat{v} = \frac{v}{\sqrt{1+ \eps^2 |v|^2}}$. Here and in what follows, we use the following notation, for any distribution function $g(t,x,v)$:
$$
\rho(g)(t,x) = \int_{\mathbb{R}^3} g \; dv, \quad j (g) (t,x)= \int_{\mathbb{R}^3} \hat{v} g \; dv.
$$
In this paper, we shall focus on equilibria that are 
\begin{itemize}
\item \emph{radial}, that is to say $\mu \equiv \mu(|v|^2)$;
\item \emph{smooth} (i.e. $C^k$, with $k \gg 1$) and  \emph{decaying sufficiently fast at infinity} (i.e. integrable against a high degree polynomial in $v$);
\item \emph{normalized} in the sense that $\int_{\mathbb{R}^3} \mu(v) \, dv =1$.
Note also, $\mu$ being radial, that  we have $
 \int_{\mathbb{R}^3}  \mu(v) \hat{v} \, dv =0,
$
for all $\eps>0$.
\end{itemize}

We study solutions of the form
\begin{equation}
\label{per}
\begin{aligned}
f^\varepsilon  &= \mu + \delta  f, \\
E^\varepsilon &= \delta E, \\
B^\varepsilon &=\delta  B,
\end{aligned}
\end{equation}
in which $\delta \ll 1$ is to be seen as a small  perturbation parameter; typically one may consider $\delta = \eps^r$, for some $r>0$.
The perturbation  $(f,E,B)$ then solves 
\begin{equation}\label{clVM-pert}
\left \{ \begin{aligned}
\partial_t f + \hat{v} \cdot \nabla_x f + (E + \varepsilon \hat{v} \times B)\cdot \nabla_v (\mu + \delta f) &= 0,
\\
\varepsilon \partial_t B + \nabla_x \times E = 0, \qquad \nabla_x \cdot E  &= \rho(f) ,
\\
- \varepsilon \partial_t E + \nabla_x \times B = \varepsilon j(f), \qquad \nabla_x \cdot B & =0, 
\end{aligned}
\right.
\end{equation}
together with initial conditions
\begin{equation}\label{initial}
f_{|t=0} = f_0, \quad E_{|t=0} = E_0, \quad B_{|t=0} = B_0.
\end{equation}
Although we do not write the dependance explicitly, $(f_0,E_0,B_0)$ may depend on $\delta, \eps$. We assume that $(f_0,E_0,B_0)$ is
{\em normalized} so that 
\begin{equation}\label{avg-initial}
\begin{aligned}
&\iint_{\T^3 \times \R^3} f_0 \, dv dx =0, \qquad \iint_{\T^3 \times \R^3} f_0 \hat{v}  \, dv dx =0,\\
&\na_x \cdot E_0 =  \int_{\R^3} {f_0} \, dv, \quad \na_x \cdot B_0 =0.
\end{aligned}
\end{equation}

\begin{remark} 
\label{rem-j}
{\em The constraint on the average of the current density could be relaxed, for instance, 
to the condition
$$
\left| \iint_{\T^3 \times \R^3} f_0 \hat{v}  \, dv dx\right| \lesssim  \delta^{a}, \qquad a>1/13.
$$}
\end{remark}


We first recall the aforementioned instability result of \cite{HKN}. 

\begin{theorem}[Instability in the non-relativistic limit; \cite{HKN}]
\label{t-classical} 
There is a class of {\bf unstable} equilibria $\mu(v)$ (see \cite{HKN} for a precise description) such that the following holds. 
For any $m,k,K,p>0$, there exist a 
family of smooth solutions $(f^{\varepsilon}, E^{\varepsilon}, B^{\varepsilon})_{\varepsilon>0}$ of \eqref{clVM1}, with $f^\varepsilon\ge 0$, and a sequence of times $t_\varepsilon = \mathcal{O}(|\log \varepsilon|)$ such that 
\begin{equation}
\| (1+| v|^2)^{\frac m2} ({f^{\varepsilon}}_{\vert_{t=0}}- \mu)\|_{H^k(\mathbb{T}^3\times \mathbb{R}^3)} \le \varepsilon^p,
\end{equation}
but 
\begin{equation}
\label{e-thm1}
\liminf_{\varepsilon \rightarrow 0}  \left\| f^\varepsilon(t_\varepsilon) - \mu\right\|_{H^{-K}(\mathbb{T}^3\times \mathbb{R}^3)} >0,
\end{equation}
\begin{equation}
\label{e-thm3}
\liminf_{\varepsilon \rightarrow 0}  \left\| E^\varepsilon(t_\varepsilon) \right\|_{L_x^2(\mathbb{T}^3)} >0.
\end{equation}
\end{theorem}


In the present paper, we consider equilibria $\mu$ that are \emph{\bf stable} in the sense of Penrose (for the Vlasov-Poisson system).
 We shall rely on  the formulation used for example  by Mouhot-Villani \cite{MV} for the study of Landau damping. This stability condition for $\mu$ reads as follows:
\begin{equation}
\label{Pen}
\inf_{\gamma>0, \tau \in \R, k \in \Z^3} \left| 1 + \int_0^{+\infty} e^{-(\gamma+i\tau)s} \widehat{\mu} ( ks) s \, ds \right| >0,
\end{equation}
where $\widehat{\mu} (\xi) = \frac{1}{(2\pi)^3} \int_{\R^3} \mu(v) e^{-i \xi \cdot v} \, dv$ is the Fourier transform of $\mu$.
Such a stability assumption is actually automatically satisfied for any radially symmetric and non-negative equilibrium $\mu$ in three (and higher) space dimensions; see \cite[Section 2.2]{MV}. This in particular includes the typical normalized Mawellian $\mathcal{M}(v) := \frac{1}{(2\pi)^{d/2}} e^{-|v|^2/2}$, and the equilibria $\mu = \mu(|v|^2)$ studied in this paper.


We also refer to  \cite{FR,BMM1,BMM2,Bed,Y1,Y2,Tri,Bed2} for other works related to Landau damping for Vlasov-Poisson equations.

\subsection{Main results}
 
Our first result in this paper is as follows. 
 
 \begin{theorem}
 \label{thm}
 Let $\mu(v)$ be a radial, smooth, fast decaying, and normalized equilibrium. 
 Let $n\geq 4$, $k>\frac{15}{2}$, and $M_0>0$.
 There are $\eps_0, \delta_0>0$ and $\lambda_0>0$, such that for all $\eps\leq \eps_0, \delta\leq \delta_0$ and 
 for all normalized data $(f_0, E_0,B_0)$ satisfying 
 $$
 \|(1+ |v|^2)^{k} f_0\|_{H^n_{x,v}} +  \|(E_0, B_0 )\|_{H^n_x} \leq M_0,
 $$
  there is a unique smooth solution  $(f^\eps,E^\eps,B^\eps)$ of the Vlasov-Maxwell system ~\eqref{clVM1}, in the form of \eqref{per}-\eqref{initial}, on the time interval $I_{\eps,\delta}: = [0,  \lambda_0 \min ( \eps^{-\alpha},  \delta^{-2/(2n+5)})]$, where $\alpha =1/2$
   if $n<6$, and $\alpha=1$ if $n\geq 6$. 
    In addition, we have 
 \begin{equation}
\label{result}
\begin{aligned}
\sup_{ [0,t]}  \| (1+ |v|^2)^{k}(f^\eps(s) - \mu) \|_{H^n_{x,v}}  &\lesssim \delta  (1+t^{n+1/2}) M_0, 
\\ \| (E^\varepsilon, B^\eps) \|_{L^2(0,t;H^n_x)}  &\lesssim \delta(1+t) M_0,
\end{aligned}
  \end{equation} for $t\in I_{\eps,\delta}$. 
   \end{theorem}

      \begin{remark}
{\em       Two general remarks are in order:}
      \begin{itemize}
 \item  {\em This is, at least to the best of our knowledge, the first instance of long time estimates for any singular limit in which the Vlasov-Poisson system is the target equation. 
As a matter of fact, polynomial times were also reached in the context of the mean field limit, see \cite{CR1, CR2}, but for the case of smoother interaction potentials. A stability condition on the equilibrium also has to be imposed in these works, as instabilities may show up in the case of unstable configurations \cite{HKN2}.
}

\item {\em For what concerns long time estimates for non-relativistic limits for other models, we refer to the recent work of Lu and Zhang \cite{LZ} in which they obtain results in polynomial times for Klein-Gordon type systems. No stability condition is necessary for initial data, as the system satisfies a kind of transparency condition and a long time WKB analysis turns out to be possible.
 }
 
 \end{itemize}
 
   \end{remark}
   
   Observe that we obtain, in Theorem~\ref{thm}, an improved order in terms of the speed of light $c =1/\eps$ (namely times of order $c$) assuming extra smoothness for the initial condition. The improvement comes from the following fact: in the proof we will have to estimate some  norms of integro-differential operators. A treatment of these using the Cauchy-Schwarz inequality yields some growth in time, which accounts for the limitation in $1/\eps^{1/2}$. It is possible though to use more elaborate tools (see \cite[Proposition 5.1 and Remark 5.1]{HKR} and Proposition~\ref{propK} below) to tame this growth, at the expense of asking for high regularity of the solutions (that corresponds precisely to the requirement $n\geq 6$). 
   
   To prove Theorem~\ref{thm}, the general idea is to rely on the fact that the Vlasov-Poisson system is a good approximation of the relativistic Vlasov-Maxwell system in the limit $\eps \to 0$. The order of the approximation is, at least formally, in $\mathcal{O}(\eps)$, which accounts for the time $1/\eps$ that we reach (the aforementioned derivations of \cite{AU,DEG,SCH} also show that on finite intervals of time, the order of convergence is $\mathcal{O}(\eps)$ --  we will not use this information in the proof though).
   We will then be able to use some linear tools devised by Mouhot and Villani \cite{MV} in the context of Landau damping for Vlasov-Poisson to reach the large times of the statement. The analysis is also inspired by the methodology introduced in \cite{HKR} for the study of the quasineutral limit of the Vlasov-Poisson system.
   
 
 \bigskip
 
It is actually possible to go beyond the scale $1/\eps$ and dramatically improve the admissible orders in terms of $1/\eps$, at the expense of considering initial data which are \emph{well-prepared} (in a sense to be defined). 
The idea is to rely on the so-called Darwin approximation of the Maxwell equations, that  corresponds to a higher order approximation than the Poisson approximation. The Vlasov-Darwin system, which we now recall, reads as follows:
\begin{equation}\label{VD} 
\left \{ \begin{aligned}
&\partial_t f^\varepsilon + \hat{v} \cdot \nabla_x f^\varepsilon + (E^\varepsilon + \varepsilon \hat{v} \times B^\varepsilon)\cdot \nabla_v f^\varepsilon =0,
\\
&E^\eps= -\na_x \phi^\eps -\eps \pa_t A^\eps, \quad B^\eps = \na_x \times A^\eps,
\\
&-\Delta_x \phi^\eps = \rho(f^\eps)-1, \\
&-\Delta_x A^\eps= \eps j(f^\eps) - \eps \partial_t \nabla_x \phi^\eps, \quad \na_x \cdot A^\eps=0.
\end{aligned}
\right.
\end{equation}  
It was studied per se in \cite{BFLS,Pal06,See,SAAI}. In \cite{BK}, Bauer and Kunze  proved that on fixed finite interval of times, the Vlasov-Darwin system~\eqref{VD} is an approximation of order  $\mathcal{O}(\eps^3)$ of the Vlasov-Maxwell system~\eqref{clVM1}. This result (which we will not rely on) supports the idea that this system is indeed a higher order approximation than Vlasov-Poisson.

 
   The procedure we follow is based on the fine structure of the linearized Vlasov-Maxwell system, and
ultimately we obtain an arbitrarily  high order approximation of the linearized Vlasov-Maxwell equations by a kind of \emph{higher-order linearized Vlasov-Darwin equations}.
Loosely speaking, compared to the standard Vlasov-Darwin system \eqref{VD}, the idea is to consider a potential vector $A^\eps$ having the form of an asymptotic expansion
$$A^\eps = \sum_{j=1}^{N} A^\eps_j, \quad \na_x \cdot A^\eps_j=0, \quad N\gg 1,$$
where $A^\eps_1$  corresponds to the usual Darwin approximation, i.e. $-\Delta_x A^\eps_1 = \eps j(f^\eps)  - \eps \partial_t \nabla_x \phi^\eps$ and
the other $A^\eps_j$ are roughly of size $\mathcal{O}(\eps^{2j+1})$ and taken as solutions of elliptic equations with sources given by \emph{higher order moments} of $f^\eps$. 
The key point is that for such Vlasov equations, we can obtain similar estimates as those of Mouhot and Villani for the linearized Vlasov-Poisson equations.
As a consequence of this theory (that we shall develop in this work), we are able to reach times with arbitrarily high order in $1/\eps$.

  
  
  
   Our result, which can be seen as the main result of this work, is gathered in the following statement.

   \begin{theorem}
   \label{thm2}
    Let $\mu(v)$ be a radial, smooth, fast decaying, and normalized equilibrium. 
    Let $N\geq 2$, $p \geq 2$,  $n\geq 6$,  $k>2N +11/2$, and $M_0>0$.
 There are $\eps_0, \delta_0>0$ and $\lambda_0>0$, such that for all $\eps\leq \eps_0, \delta\leq \delta_0$ and 
 for all normalized data $(f_0, E_0,B_0)$ that are {\bf well-prepared of order $p$} and satisfying 
 $$
 \|(1+ |v|^2)^{k} f_0\|_{H^n_{x,v}} +  \|(E_0, B_0 )\|_{H^n_x} \leq M_0,
 $$
  there is a unique smooth solution  $(f^\eps,E^\eps,B^\eps)$ of the Vlasov-Maxwell system ~\eqref{clVM1}, in the form of \eqref{per}-\eqref{initial}, on the time interval $I_{\eps,\delta}: = [0,  \lambda_0 \min ( \eps^{-\min \{N+1/2,p\}},  \delta^{-1/(n+3)})]$.
    In addition, we have 
 \begin{equation}
\label{result2}
\begin{aligned}
&\sup_{ [0,t]}  \| (1+ |v|^2)^{k}(f^\eps(s) - \mu) \|_{H^n_{x,v}}  \lesssim \delta  (1+t^{n+1/2}) M_0,
\\ &\| (E^\varepsilon, B^\eps) \|_{L^2(0,t;H^n_x)}  \lesssim \delta (1+ \eps^{p/2-1}\sqrt{t} + \eps\delta (1 + t^{n+5/2})  + \eps^{2N}(1+t^{2})   ) M_0,
\end{aligned}
  \end{equation} for $t\in I_{\eps,\delta}$. 
   \end{theorem}

   This theorem asks that the initial data is well-prepared (of order $p$);we chose to postpone the precise definition of this notion to Definition~\ref{def-wp} in Section~\ref{subsec-wp} since we need additional material for it. Loosely speaking, we ask that 
   $$\left( B_0,
E_0-   \na_x (-\Delta)^{-1} \left(\int_{\R^3} f_0 \,dv\right) \right)$$
  satisfies a high order expansion (in terms of $\eps$).
    As we shall see, for a given initial condition $(f_0,E_0,B_0)$, it might be difficult to check this assumption in practice, for high values of $p$. However, it is possible, at least for small values of $p$, to easily write down what it means, see \eqref{eq-4} for $p=4$, \eqref{eq-6} for $p=6$, \eqref{eq-8} for $p=8$.

        \begin{remark}
{\em  The linear estimates given in Theorem~\ref{thm2} can be interpreted as linearized Landau damping for the Vlasov-Maxwell system on $\T^3$, for times of order $\mathcal{O}(c^p)$,  for initial data that are well-prepared of order $p$ (for any $p \in \N$). }

      \end{remark}

The rest of the paper is  dedicated to the proof of Theorems~\ref{thm} and~\ref{thm2}.   
We  will start with Theorem~\ref{thm}, focusing first on the general case $n\ge 4$ in Sections~\ref{Sec-2} to~\ref{Sec-6}, and then indicate in  Section~\ref{sec-high} the required modifications in the case $n\geq 6$ in order to get the improved order in terms of the speed of light. Finally we handle in Section~\ref{sec-D} the case of Theorem~\ref{thm2} by developing the aforementioned higher-order linearized Vlasov-Darwin approximation.
We end the paper with two appendices where we discuss scaling invariances of the Vlasov-Maxwell system and the radial assumption for the equilibrium.



\section{Preliminaries}
\label{Sec-2}

The fields $E$ and $B$ can be constructed thanks to the electromagnetic potentials $(\phi,A)$ via the relations
\begin{equation}\label{potentials-cl}
 E = - \nabla \phi - \varepsilon \partial_t A, \qquad B = \nabla \times A,\end{equation}
 with $A$ satisfying the \emph{Coulomb gauge}
 $$
 \nabla \cdot A = 0.
 $$
The scalar and vector potentials $\phi, A$ are asked to solve
\begin{equation}\label{ellip-pert} 
\begin{aligned}
&-\Delta \phi= \rho(f), \quad \int \phi \, dx =0, \\
&\varepsilon^2 \partial_t^2 A - \Delta A= \varepsilon j(f)- \varepsilon \partial_t \nabla \phi,
 \end{aligned}
 \end{equation}
together with an initial condition $(A_{|t=0}, \pa_t A_{|t=0})$, chosen so that 
$$
B_0 = \nabla \times A_{|t=0}, \quad E_0 = \nabla \Delta^{-1} \rho(f_{|t=0}) - \varepsilon \partial_t A_{|t=0}.
$$
Note that without loss of generality, we can impose the normalization
\begin{equation}
\label{norm-A}
\int_{\T^3} A_{\vert_{t=0}} \, dx = 0.
\end{equation}
As a consequence, $(E,B)$  obtained with  \eqref{potentials-cl} and \eqref{ellip-pert} satisfy the Maxwell equations in the system  \eqref{clVM1}.

We note that in view of standard energy estimates for the wave equation (see in particular~\eqref{dtA} below), the term $\varepsilon \partial_t A$ is not small in terms of $\eps$. As in \cite{HKN}, this  motivates us to introduce the shifted distribution function 
\begin{equation}
\label{def-g}
g(t,x,v) := f(t,x,v)- \eps A(t,x) \cdot \na_v \mu(v).
\end{equation}
As a matter of fact $\eps A$ is not small either, but of order $1$. However, this shift allows to replace a kind of reaction term of order $1$ by a small term in the Vlasov equation.

It is clear that $\rho(f) = \rho(g)$, since 
$$
\int_{\R^3} A \cdot \na_v \mu(v) \, dv = 0.
$$
As for $j(f)$, we compute 
$$
\begin{aligned}
\int_{\R^3}  \hat{v} A \cdot \na_v \mu(v) \, dv &= -  \int_{\R^3} \sum_{j=1}^3 (\pa_{v_j}   \hat{v} ) A_j  \mu(v) \, dv  \\
&= - A  \int_{\R^3}  \frac{1+ \frac{2}{3} \eps^2 |v|^2}{(1+ \eps^2 |v|^2)^{3/2}} \mu(v) \, dv
\\&=: - \lambda(\mu,\e) A
\end{aligned}
$$
which gives a non-trivial contribution.  This yields 
\begin{equation}\label{j-g}j(f) = j(g) - \e \lambda(\mu,\e) A  \end{equation}
in which we note that for all $\e>0$, $1\leq \lambda(\mu,\e) \leq 5/3$.

The shifted distribution function $g$ then satisfies the equation 
\begin{equation}\label{main}
\partial_t g + \hat{v} \cdot \nabla_x g + \delta (E + \varepsilon \hat{v} \times B)\cdot \nabla_v   g  - \na_x \phi \cdot \na_v \mu = R
\end{equation}
with the remainder $R$ defined by 
\begin{equation}
\label{def-R}
\begin{aligned}
R:= 
&-\eps \hat{v} \cdot \na_x (A \cdot \na_v \mu) 
- \delta \eps (E+ \eps \hat{v} \times B) \cdot \na_v(A \cdot \na_v \mu).
\end{aligned}
\end{equation}
Note that we used the fact that $(\hat{v} \times B) \cdot \na_v \mu= 0$ since $\mu$ is radial. 
This remainder is expected to be small, or at least to be controlled in polynomially growing times in terms of $\delta$ and $\eps$.

The scalar and vector potentials $\phi, A$ satisfy as well
\begin{equation}\label{ellip-pert2} 
-\Delta \phi= \rho(g), \qquad \varepsilon^2 \partial_t^2 A - \Delta A + \eps^2 \lambda(\mu,\e) A  = \varepsilon j(g)- \varepsilon \partial_t \nabla \phi .\end{equation}

\section{Set up of the bootstrap argument}

%

We set up a bootstrap argument. We introduce the key norm for our analysis
\begin{equation}
\label{def-N}
\Nc (t) := \|(\rho(g),j(g))\|_{L^2(0,t; H^n_x)} 
\end{equation}
for some integer $n>\frac 72$, which is fixed until the end of Section~\ref{Sec-5}.


Let $M>1$ be a large real number to be fixed later, independently of $\eps$ and $\delta$. We impose in particular that 
$\Nc(0)  < M$.
 Let 
\begin{equation}
\label{def-T}
T_\eps := \sup\Big \{T\geq 0, \, \Nc(T) \leq M\Big\}.
\end{equation}
By the standard local existence theory for Vlasov-Maxwell, we already have that $T_\eps>\tau_0$, for some $\tau_0 >0$\footnote{This follows from instance from \cite[Proposition 3.2]{HK} combined with the uniform (with respect to $\eps$) estimates for the electromagnetic  field contained in~\eqref{dtA} and~\eqref{EB-est0}}.
If $T_\eps = +\infty$, there is nothing to do; see Section \ref{Sec-6} directly. We therefore assume in the following that $T_\eps$ is finite.

The goal from now on is to prove the following key proposition:

\begin{prop}
\label{key-prop}
Assume all requirements of the statement of Theorem~\ref{thm}.
There exists $M>0$ so that the following holds.
There are $\eps_0, \delta_0>0$ and $\lambda_0>0$, such that for all $\eps\leq \eps_0, \delta\leq \delta_0$,
$$
T_\eps \geq \lambda_0  \min (  \eps^{-1/2},  \delta^{-2/(2n+5)}),
  $$
  where $T_\eps$ is defined in \eqref{def-T}.
\end{prop}

In the following, we study the Vlasov-Maxwell system on the interval of time $[0,T_\eps]$ and shall rely on a kind of bootstrap/continuity argument.
In the proofs, we use the symbol $\lesssim$ for inequalities $A \lesssim B$, which will systematically mean that there is $C>0$ independent of $\eps, \delta$ and $M$ such that $A \leq C B$. For the estimates, we will also need the fact that $\eps \leq \eps_0$ and $\delta\leq \delta_0$ for $\eps_0, \delta_0$ small enough; however we will track down this dependance explicitly.

 Let us start with estimates for the electromagnetic potentials  and  fields $\phi, A, E,B$ on $[0,T_\eps]$.
 
 \begin{lemma}
 \label{EB}
 For $t \in [0,T_\eps]$, we have
  \begin{align}
  \label{phi-est}
  \| \na_x \phi \|_{L^2(0,t; H^n_x)} &\leq M, \\
  \label{A-est}
\| (\e A, \e\partial_t A, \nabla_x A) (t)\|_{L^\infty(0,t; H^n_x)}  &\lesssim (1+t^{1/2})M, \\
  \label{EB-est}
\| (E,B)\|_{L^2(0,t; H^n_x) }  &\lesssim  (1+t)M. 
\end{align}
 \end{lemma}
 
 \begin{proof}[Proof of Lemma~\ref{EB}]
The estimate~\eqref{phi-est} follows from  standard elliptic estimates for the Poisson equation~\eqref{ellip-pert} satisfied by $\phi$, with the zero average source $\rho(g)$, and the bound $\| \rho \|_{L^2(0,t; H^n_x)}  \leq M$ on $[0,T_\eps]$.

For what concerns~\eqref{A-est}, we proceed with standard Sobolev energy estimates for the wave equation \eqref{ellip-pert2} on the vector potential $A$; this yields
\begin{equation}\label{est-phiA}\begin{aligned}
 \frac12 \frac{d}{dt} \Big( \| \varepsilon \partial_t A\|_{H^n_x }^2 &+ \| \nabla_x A\|_{H^n_x}^2 +  \lambda(\mu,\e) \| \e A \|_{H^n_x}^2 \Big) 
\\ & \le \| \varepsilon \partial_tA\|_{H^n_x } \Big( \| j(g)\|_{H^n_x } + \|\partial_t \nabla \phi\|_{H^n_x }\Big).%
 \end{aligned}\end{equation}
We have the local conservation of charge
 $$
 \partial_t \rho(f) + \na_x \cdot j(f) = 0,
 $$
obtained by integrating \eqref{clVM-pert} in velocity, which together with \eqref{j-g} yields  for the shifted distribution function $g$:
$$
 \partial_t \rho(g) + \na_x \cdot j(g) = \e \lambda(\mu,\e) \na_x \cdot A =0,
$$
since $A$ satisfies the Coulomb gauge. Differentiating with respect to time  the Poisson equation for $\phi$, we thus get
$$
-\Delta_x \partial_t \phi =  - \na_x \cdot j(g),
$$
from which we deduce using standard elliptic estimates that
$$
 \|\partial_t \nabla_x \phi\|_{H^n_x } \lesssim \| j(g)\|_{H^n_x }.
$$
Injecting these into \eqref{est-phiA}, we thus get 
$$
 \frac{1}{2}\frac{d}{dt} \Big( \| \varepsilon \partial_t A\|_{H^n_x }^2 + \| \nabla_x A\|_{H^n_x}^2 +  \lambda(\mu,\e) \| \e A \|_{H^n_x}^2 \Big)  \lesssim  \| \varepsilon \partial_tA\|_{H^n_x } \| j(g)\|_{H^n_x },
$$
which implies 
$$
 \frac{d}{dt}  \| ( \varepsilon \partial_t A, \nabla_x A, \sqrt{\lambda(\mu, \e)}\e A) (t)\|_{H^n_x }  \lesssim  \| j(g)\|_{H^n_x } .
$$
Since $1\le \lambda(\mu,\e) \le 5/3$, the above yields 
\begin{equation}
\label{dtA}
 \| (\e A, \varepsilon \partial_t A, \nabla_x A)(t)\|_{L^\infty(0,t;H^n_x)} \lesssim \int_0^t \| j(g)\|_{H^n_x } \, d\tau + \|(\eps A, \eps \partial_t A, \na A)_{|t=0} \|_{H^n_x } .
\end{equation}
%
Taking $M$ large enough so that 
$$\|(\eps A, \eps \partial_t A, \na A)_{|t=0}  \|_{H^n_x } \lesssim \| (E_0,B_0)\|_{H^n_x} + \|  f_0 \, dv\|_{H^n_k} \le M,$$
we obtain at once \eqref{A-est}, upon using 
$$
 \int_0^t \| j(g) \|_{H^n_x}\; d\tau \leq t^{1/2} \| j(g) \|_{L^2(0,t; H^n_x)} \le t^{1/2} M.
$$ 
Next, looking at the definition of the electromagnetic field $(E,B)$ in terms of the electromagnetic potentials, and using again the elliptic estimate $ \| \nabla_x \phi\|_{H^n_x } \lesssim \| \rho(g)\|_{H^n_x }$ from the Poisson equation, 
we get 
\begin{equation}\label{EB-est0}\begin{aligned}
\| E \|_{L^2(0,t;H^n_x)}  & \lesssim \| \rho(g)\|_{L^2(0,t;H^n_x)} + \| \varepsilon \partial_t A\|_{L^2(0,t;H^n_x)} \le  (1+ t) M,\\
\| B \|_{L^2(0,t;H^n_x) }  & \le \| \nabla_x A\|_{L^2(0,t;H^n_x) }  \lesssim (1+t )M
,
\end{aligned}\end{equation}
for $t \in [0,T_\eps]$. This proves \eqref{EB-est}. 
 \end{proof}

%
 
%
%
 
 
 
 \section{Weighted Sobolev bounds}
 
 For all $n,k \in \N$, let $H^n_k$ be the Sobolev space with polynomial weights in velocity, associated with the norm
$$
\| \cdot \|_{H^n_k} = \sum_{|\alpha|+|\beta| \le n}\left( \int_{\T^3}\int_{\R^3} | \partial_x^\alpha \partial_v^\beta (\cdot) |^2  (1+ |v|^2)^k \, dv dx \right)^{1/2}.
$$
We will also use the notation $L^2_k := H^0_k$.
The main goal of this section is to derive some $L^\infty_t H^n_k$ bounds for $g$ on $[0,T_\eps]$.

%
%
%
%
%
%
%

\begin{lem}
\label{prem}
Let $k>15/2$ and an integer $n>\frac72$. There is $\lambda_0>0$, such that for all $t \leq \min (T_\eps, \lambda_0 \delta^{-2/(2n+3)})$, we have
\begin{align}
\label{eq-h}
\sup_{[0,t ]}\| g \|_{H^n_k} \lesssim (1+t^{n+1/2}) M  + \eps (1+t^{n+3/2}) M. 
\end{align}
\end{lem}

\begin{proof}[Proof of Lemma~\ref{prem}] The proof follows from standard high-order energy estimates (as done in \cite{HKN}), making use of the ``triangular'' structure in the system satisfied by the derivatives of $g$. Precisely, we shall derive weighted $L^2$ estimates for $\pa^\alpha_x \pa^\beta_v g$, with $|\alpha |+ |\beta|=n$,  starting from $|\alpha|= n$ down to $|\alpha| =0$. Setting for convenience $F(t,x,v) = E(t,x) + \eps \hat{v} \times B(t,x)$, the function $\pa^\alpha_x \pa^\beta_v g$ solves the following transport equation
\begin{equation}\label{main-derv}
\begin{aligned}
\Big(\partial_t  + \hat{v} \cdot \nabla_x  + \delta F\cdot \nabla_v   \Big) \pa^\alpha_x \pa^\beta_v g 
&=  \pa^\alpha_x \pa^\beta_vR + \na_x  \pa^\alpha_x  \phi \cdot \na_v  \pa^\beta_v \mu 
\\&\quad  - [ \pa^\beta_v , \hat{v} \cdot \nabla_x ] \pa^\alpha_x g - \delta [ \pa^\alpha_x \pa^\beta_v,  F\cdot \nabla_v ] g 
\end{aligned}
\end{equation}
in which the bracket denotes the usual commutator terms:
$$\begin{aligned}
~[ \pa^\beta_v , \hat{v} \cdot \nabla_x ] \pa^\alpha_x g &= \sum_{\gamma_1 + \gamma_2 = \beta, \gamma_2 \not = \beta} \pa^{\gamma_1}_v \hat{v} \cdot \nabla_x  \pa_v^{\gamma_2}\pa^\alpha_x g,
\\
~ [ \pa^\alpha_x \pa^\beta_v,  F\cdot \nabla_v ] g &= \sum_{\gamma_1 + \gamma_2 = \alpha + \beta, \gamma_2 \not = \alpha + \beta}  \partial_{x,v}^{\gamma_1}F\cdot \nabla_v \partial_{x,v}^{\gamma_2} g.
\end{aligned}$$
Standard ($v$-weighted) $L^2$ estimates yield 
$$
\begin{aligned}
\frac{d}{dt}\Big \|\pa^\alpha_x \pa^\beta_v g \Big \|_{L^2_k}
& \lesssim \Big\|  \pa^\alpha_x \pa^\beta_vR  + \na_x  \pa^\alpha_x  \phi \cdot \na_v  \pa^\beta_v \mu 
 - [ \pa^\beta_v , \hat{v} \cdot \nabla_x ] \pa^\alpha_x g
- \delta [ \pa^\alpha_x \pa^\beta_v,  F\cdot \nabla_v ] g \Big\|_{L^2_k} \\&\lesssim A_1(t) + A_2(t) + A_3(t) + A_4(t),
\end{aligned}
$$
in which $A_j(t)$ denotes the weighted $L^2$ norm of each term of the above sum. 

~\\
\noindent \emph{Term $A_1$.} Recall that the remainder $R$ is defined by 
$$
\begin{aligned}
R:= 
&-\eps \hat{v} \cdot \na_x (A \cdot \na_v \mu) 
- \delta \eps (E+ \eps \hat{v} \times B) \cdot \na_v(A \cdot \na_v \mu).
\end{aligned}
$$
Relying on the rapid decay of $\nabla_v \mu$ at infinity to absorb polynomial weights in $v$ and using the bounds $|\eps\hat v | \leq 1$ and $|\pa^\alpha_v \hat v | \lesssim 1$, for $\alpha\neq 0$, we obtain 
$$
\begin{aligned}
\| R(t)\|_{H^n_k}
&\lesssim  \| \nabla_v \mu \|_{H^{n+1}_{k+1}}   \Big [ \eps\|  \na_x A(x) \|_{H^{n}_x} + \delta  \| (E,B) |\eps A|\|_{H^n_x} \Big ] \\
&\lesssim
\eps  \|  \na_x A\|_{H^n_x} + \delta  \| (E,B)\|_{H^n_x}  \| \eps A \|_{H^n_x},
\end{aligned}
$$
in which the last line used the algebra structure of the Sobolev space $H^n(\T^3)$, since $n>\frac32$.  
By Lemma~\ref{EB}, we end up with
\begin{equation}\label{bd-Rem}
\int_0^t A_1(s) \, ds = \int_0^t \|  \pa^\alpha_x \pa^\beta_vR(s)\|_{L^2_k} \, ds  \lesssim  \eps (1+t^{3/2}) M + \delta (1+t^2) M^2 
\end{equation}
for all $t \in [0,T_\eps]$.

~\\
\noindent \emph{Term $A_2$.}  Likewise, thanks again to Lemma \ref{EB}, we obtain
$$ \int_0^t A_2(s) \; ds\le \int_0^t \| \na_x  \pa^\alpha_x  \phi \cdot \na_v  \pa^\beta_v \mu\|_{L^2_k}  \lesssim \int_0^t \| \na_x \phi \|_{H^n_x} \lesssim (1+\sqrt t)M.$$

\noindent \emph{Term $A_3$.}  Clearly $A_3 = 0$ when $\beta=0$. For $\beta \not =0$, since $|\pa^{\gamma_1}_v \hat v | \lesssim 1$ for $\gamma_1 \neq 0$, we have 
$$ A_3 = \sum_{\gamma_1 + \gamma_2 = \beta, \gamma_2 \not = \beta} \|\pa^{\gamma_1}_v \hat{v} \cdot \nabla_x  \pa_v^{\gamma_2}\pa^\alpha_x g \|_{L^2_k} \lesssim \| \pa_x^{\alpha }\pa_v^{\beta-1} \na_x g \|_{L^2_k}.$$
~


\noindent \emph{Term $A_4$.}  
For  $\gamma_1 + \gamma_3 = \alpha$, $\gamma_2 + \gamma_4 = \beta$, with $|\gamma_3|+ |\gamma_4| < n$, we estimate
$$
\begin{aligned}
\| \partial_{x,v}^{\gamma_1,\gamma_2}(E+ \eps \hat{v} \times B)\cdot \nabla_v \partial_{x,v}^{\gamma_3,\gamma_4} g\|_{L^2_k}
&\lesssim
 \| \partial_{x}^{\gamma_1}(E, B) \partial_{x}^{\gamma_3}\pa_v^{\gamma_4} \na_v g\|_{L^2_k}.
\end{aligned}
$$
For $|\gamma_1|\le 2$, since $n > \frac72$, the standard Sobolev embeding over $\T^3$ yields 
$$ \begin{aligned}
\| \partial_{x}^{\gamma_1}(E, B) \partial_{x}^{\gamma_3}\pa_v^{\gamma_4} \na_v g\|_{L^2_k} 
&\lesssim \| (E,B)\|_{W^{|\gamma_1|,\infty}_x} \|  \partial_{x}^{\gamma_3}\pa_v^{\gamma_4} \na_v g\|_{L^2_k}
\\
&\lesssim \| (E,B)\|_{H^n_x} \| g\|_{H^n_k},
\end{aligned}$$
since $|\gamma_3|+|\gamma_4|+1 \le n$. Similarly, for $|\gamma_1|\ge 3$, we have $|\gamma_3|\le |\alpha|-3$ and 
$$ \begin{aligned}
\| \partial_{x}^{\gamma_1}(E, B) \partial_{x}^{\gamma_3}\pa_v^{\gamma_4} \na_v g\|_{L^2_k} 
&\lesssim \| (E,B)\|_{H^{|\gamma_1|}_x} \|  \partial_{x}^{\gamma_3}\pa_v^{\gamma_4} \na_v g (1+|v|^2)^{k/2}\|_{L^2_vL^\infty_x}
\\
&\lesssim \| (E,B)\|_{H^{|\gamma_1|}_x} \| \pa_v^{\gamma_4} \na_v g (1+|v|^2)^{k/2}\|_{L^2_vH^{|\alpha|-1}_x}
\\
&\lesssim \| (E,B)\|_{H^{n}_x} \| g\|_{H^n_k}.
\end{aligned}$$
These prove 
$$\int_0^t A_4(s)\; ds \lesssim \delta \int_0^t \| (E,B)(s)\|_{H^{n}_x} \| g(s)\|_{H^n_k} \; ds \lesssim \delta (1+ t^{3/2})M \sup_{[0,t]} \| g(s)\|_{H^n_k}.$$

To conclude, in the case when $\beta =0$ (in which case there is no $A_3$ term), we obtain 
$$
\begin{aligned}
\sup_{[0,t]}\|\pa^\alpha_x g \|_{L^2_k}
& \lesssim \eps (1+t^{3/2}) M + \delta (1+t^2) M^2+ (1+\sqrt t)M 
\\&\quad + \delta (1+ t^{3/2})M \sup_{[0,t]} \| g(s)\|_{H^n_k}
=: \mathcal{B}(t)
\end{aligned}
$$
for all $\alpha$ so that $|\alpha|\le n$. Next, we proceed with the case when $|\beta|=1$, and $|\alpha|\le n-1$. In this case, $A_3 \lesssim \| \na_x^{|\alpha| +1} g\|_{L^2_k}$, which has been already estimated in the previous step. We thus obtain 
$$
\begin{aligned}
\sup_{[0,t]}\|\pa^\alpha_x \pa_v^\beta g \|_{L^2_k}
& \lesssim \mathcal{B}(t) + \int_0^t  \| \na_x^{|\alpha| +1}  g(s)\|_{L^2_k}\; ds \le (1+t )\mathcal{B}(t)
\end{aligned}
$$
for $|\beta|=1$. By induction, we obtain for $|\alpha|\le n- \ell$ and $|\beta|=\ell$, 
$$
\sup_{[0,t]} { \| \pa^\alpha_x \pa^\beta_v  g \|_{L^2_k} } \lesssim (1+t^\ell) \mathcal{B}(t), 
$$
which proves 
\begin{equation}
\sup_{[0,t]} \|  g \|_{H^n_k} \lesssim (1+t^n) \mathcal{B}(t). 
\end{equation}
In view of the definition of $\mathcal{B}(t)$, there is $\lambda_0>0$, small enough such that, for $t \leq \min(T_\eps, \lambda_0 \delta^{-2/(2n+3)})$, we can
simplify the estimate by absorbing $\sup_{[0,t]} { \|  g \|_{{H}^n_k}}$ on the right hand side. By noting in particular that $\delta (1+t^2) \lesssim 1$, this yields 
$$
\sup_{[0,t]} { \|  g \|_{{H}^n_k} } \lesssim  (1+t^{n+1/2}) M + \eps (1+t^{n+3/2})M,
$$
as claimed.
\end{proof}

\begin{remark}
{\em It may have been natural to introduce the distribution function  
$$h (t,x,v)= g(t,x+t\hat v,v)$$
and derive high-order estimates directly on $h$. Indeed, $h$ satisfies 
\begin{equation}
\begin{aligned}
\partial_t h + \delta\big(F(t,x+t\hat v,v) \cdot \na_v - t  F(t,x+t\hat v,v) \cdot (\na_v \hat v \cdot \nabla_x)\big) h  \\ =  \na_x \phi (t,x+ \hat v) \cdot \na_v \mu +    R(t,x+t \hat v,v),  
\end{aligned}
\end{equation}
denoting $F(t,x,v) = E(t,x) + \eps \hat{v} \times B(t,x)$. However, because of the growth in time of several of the source terms in the equation, such an approach would yield extra growth in time in the final estimates, when translating these back in terms of $g$.
}\end{remark}

Thanks to the estimates of Lemma~\ref{prem}, when applying $x$ derivatives to the transport equation satisfied by $g$, the contribution of $R$ defined as in~\eqref{def-R}~can indeed be seen as a remainder, and so is that of all commutators, as shown by the next result.

\begin{lem}
\label{commu}
For all $\alpha$ such that $|\alpha|\le n$, we have
$$
\partial_t \pa_x^{\alpha} g + \hat{v} \cdot \nabla_x \pa_x^{\alpha} g + \delta (E + \varepsilon \hat{v} \times B)\cdot \nabla_v   \pa_x^{\alpha} g  - \pa_x^\alpha \na_x \phi \cdot \na_v \mu = R_\alpha
$$
with 
$$
R_\alpha :=\pa_x^{\alpha} R - \delta [ \pa^\alpha_x,   (E + \eps \hat{v} \times B)\cdot \nabla_v ] g.
$$
Moreover, for all $t \leq \min (T_\eps, \lambda_0 \delta^{-2/(2n+3)})$, we have
\begin{equation}
\label{R1}
\begin{aligned}
\| R_\alpha \|_{L^2(0,t; L^2_k)} \lesssim  \delta (1+t^{n+3/2}) M^2 + \e (1 + t)M 
.
  \end{aligned}
\end{equation}
\end{lem}
 \begin{proof}[Proof of Lemma~\ref{commu}]
The remainder term $R_\alpha$ consists of $A_1$ and $A_4$ terms (keeping the same notations as in the proof of Lemma~\ref{prem}). As already estimated in the previous lemma, we have 
$$\int_0^t \|  \pa^\alpha_x R(s)\|^2_{L^2_k} \, ds  \lesssim  \eps^2 (1+t^2) M^2 + \delta^2 (1+t^{3}) M^4$$
and for $\beta+ \gamma = \alpha, \, |\gamma| \neq |\alpha|$, 
$$
\begin{aligned}
 \int_0^t \| (\pa_x^\beta E + \varepsilon \hat{v} \times \pa_x^\beta B) \cdot \na_v \pa_x^\gamma g\|_{L^2_k}^2 \; ds
 &\lesssim  \int_0^t \| (E,B)(s)\|^2_{H^{n}_x} \| g(s)\|^2_{H^n_k} \; ds
 \\& \lesssim (1+t^2)M^2 \sup_{[0,t]} \| g(s)\|_{H^n_k}^2.
 \end{aligned}$$
Using the bound on $g$ in Lemma \ref{prem}, we obtain at once
$$\| R_\alpha \|_{L^2(0,t; L^2_k)} \lesssim  \delta (1+t^{n+3/2}) M^2 + \e (1 + t)M 
+ \delta \eps (1+t^{n+5/2}) M^2,$$
 in which the last term can be directly bounded by the second term on the right, since $t \leq \min (T_\eps, \lambda_0 \delta^{-2/(2n+3)})$. The lemma follows. 
 \end{proof}

\section{Electrostatic Penrose stability and consequences}
\label{Sec-5}
We now aim at studying $L^2$ estimates for the moments $\rho(G)$ and $j(G)$ associated to a solution $G$ (which has to be thought of as $\pa^\alpha_x g$, $|\alpha|=n$) of the
linearized equation
\begin{equation}\label{main1}
\partial_t G + \hat{v} \cdot \nabla_x G + \delta (E + \varepsilon \hat{v} \times B)\cdot \nabla_v   G  - \na_x \phi[G] \cdot \na_v \mu = \mathcal{R},
\end{equation}
with $\phi[G]$ solving 
$$
-\Delta_x \phi[G] = \int_{\R^3} G \, dv.
$$
Here, $\mathcal{R}$ satisfies the same estimate as in~\eqref{R1}, that is 
\begin{equation}
\label{R1'}
\begin{aligned}
\| \mathcal{R} \|_{L^2(0,t; L^2_k)} \lesssim  \delta (1+t^{n+3/2}) M^2 + \e (1 + t)M,
  \end{aligned}
\end{equation}
and $(E,B)$ stands for the electromagnetic field  associated to the distribution function $f$. 

\subsection{Straightening characteristics}

A first step consists, as in \cite{HKR}, of performing the change of variables $v \mapsto \Phi(t,x,v)$ in order to  straighten characteristics to go from \eqref{main1} to
\begin{equation}\label{main2}
\partial_t F + \hat \Phi(t,x,v) \cdot \nabla_x F   - \na_x \phi[G] \cdot \na_v \mu (\Phi(t,x,v)) =  \mathcal{R}',
\end{equation}
where $\hat \Phi = \frac{\Phi}{\sqrt{1 + \eps^2 |\Phi|^2}}$, and  
$$
F (t,x,v) := G (t,x, \Phi(t,x,v)), \quad
 \mathcal{R}' (t,x,v) := \mathcal{R}(t,x, \Phi(t,x,v)).
$$
To achieve this, we define $\Phi(t,x,v)$ as solving the Burgers equation
\begin{equation}\label{Burgers}
\left\{
\begin{aligned}
&\pa_t \Phi + \hat \Phi \cdot \na_x \Phi = \delta (E + \varepsilon \hat \Phi \times B), \\
&\Phi(0,x,v) = v.
     \end{aligned}
     \right.
\end{equation}
It is straightforward to check that such a change of variables indeed allows to obtain \eqref{main2}:
\begin{align*}
\partial_t F + \hat \Phi(t,x,v) \cdot \nabla_x F  &= (\partial_t G + \hat{v} \cdot \nabla_x G ) (t,x, \Phi(t,x,v)) \\
&\quad + (\pa_t \Phi + \hat \Phi \cdot \na_x \Phi) \cdot \na_v G (t,x,\Phi(t,x,v)) \\
&= (\partial_t G + \hat{v} \cdot \nabla_x G +  (E + \varepsilon \hat v \times B) \cdot \na_v G ) (t,x, \Phi(t,x,v)) \\
&= \na_x \phi[G] \cdot \na_v \mu (\Phi(t,x,v)) +  \mathcal{R}(t,x, \Phi(t,x,v)).
\end{align*}

In this section, we shall always use the notation $\hat{\cdot} = { \cdot \over \sqrt{ 1  +  \eps^2  | \cdot |^2} } $.
We have the following deviation estimates for $\Phi$, inspired from  \cite[Lemma 4.6]{HKR}.

\begin{lem}
\label{lem-phi}
There is $\lambda_1>0$ depending only on $M$ such that the following holds.
We have for all $t \leq min(T_\e, \lambda_1\delta^{-2/5})$ that there exists a unique solution $\Phi(t,x,v) \in C^0([0,t]; W^{2,\infty}_{x,v})$ of \eqref{Burgers} and 
the following estimates hold. 

In addition, for all $t \leq \min(T_\e, \lambda_1  \delta^{-2/5})$, we have
\begin{align}
 \label{estim-Burgers-infini}
  \sup_{[0,t]} \| \Phi  - v\|_{W^{2,\infty}_{x,v}}  \lesssim  \delta (1+t^{5/2})M, \\
 \label{estim-Burgers-infini-bis}
  \sup_{[0,t]}\left\|( 1 + \eps^2 |v|^2)^{1 \over 2} (\hat{\Phi}  - \hat{v})\right\|_{W^{2,\infty}_{x,v}}  \lesssim  \delta (1+t^{5/2})M.
\end{align}
%
\end{lem}

\begin{proof}[Proof of Lemma~\ref{lem-phi}]
  We focus on the proof of 
  \eqref{estim-Burgers-infini}, since the existence and uniqueness of $\Phi$ then follows by standard arguments.
  
  Set $\eta= \Phi -v$. Then $ \eta$ solves
     $$  \pa_t  \eta + \widehat{v+\eta} \cdot \nabla_{x}  \eta  = \delta (E + \varepsilon \widehat{(v + \eta)} \times B)$$
      with zero initial data. 
The $L^\infty$ bound for $\eta$ is straightforward  from the $L^\infty$ estimate for the transport equation, Lemma~\ref{EB} and the Sobolev embedding in the $x$ variable: 
$$
\begin{aligned}
\| \eta \|_{L^\infty_{x,v}} &\lesssim \delta \int_0^t \| (E,B)\|_{L^\infty} \, ds \lesssim \delta (1+t^{3/2}) M.
\end{aligned}
$$
We focus more on the estimate for the derivatives.
Taking derivatives of the equation yields 
\begin{equation}\label{eq-comm}
\begin{aligned}
\pa_t \pa_x \Phi +  \widehat{\Phi}\cdot \na_x \pa_x  \Phi  = &\delta\pa_x( E + \eps\widehat{\Phi} \times B)  - \partial_x \widehat{\Phi} \cdot \nabla_{x}  \Phi
\\
\pa_t \pa_v \eta +  \widehat{v+\eta}\cdot \na_x \pa_v  \eta  = &\delta( E + \eps\pa_v\widehat{(v+\eta)} \times B)  - \partial_v \widehat{(v+\eta)} \cdot \nabla_{x}  \eta.
\end{aligned}\end{equation}
Note that $\partial \widehat{(v+\eta)} = \frac{\partial (v + \eta) }{\sqrt{1+\eps^2 |v+\eta|^2}} - \frac{\eps^2 (v + \eta) (v + \eta) \cdot \partial (v+\eta)}{(1+\eps^2 |v+\eta|^2)^{3/2}} $, and so 
$$ \| \partial_x \widehat{\Phi} \|_{L^\infty} \le 2 \| \partial_x \Phi\|_{L^\infty}, \qquad \| \partial_v \widehat{(v+\eta)} \|_{L^\infty} \le 2 (1+\| \partial_v \eta\|_{L^\infty}).$$
We first focus on the $x$ derivative.   Using $L^\infty$ estimates for the  transport operator and noting that $\eps |\widehat{\Phi}| \le 1$, we obtain from \eqref{eq-comm} that
 $$ 
 \begin{aligned}\| \na_x \eta (t) \|_{L^{ \infty}_{x,v}}&= \| \na_x\Phi (t) \|_{L^{ \infty}_{x,v}}\\
 &\lesssim \int_0^t  \Big( \| \na_x\Phi (s) \|_{L^{\infty}_{x,v}}^2  + \delta  (1+ \eps \| \na_x\Phi(s)\|_{L^{\infty}_{x,v}})\| (E,B)(s) \|_{W^{1, \infty}}  \Big)\, ds.
 \end{aligned}
 $$
By Lemma~\ref{EB} and Sobolev embedding in the $x$ variable we get that
$$    \int_0^t    \| (E,B)(s) \|_{W^{1, \infty}} \, ds \lesssim  (1 + t^{3/2})M .$$
 Hence, we have
 $$
 \sup_{[0,t]}   \| \na_x\Phi  \|_{L^{\infty}_{x,v}} \lesssim t   \sup_{[0,t]}   [\| \na_x\Phi  \|_{L^{\infty}_{x,v}} ]^2 + \delta (1+ t^{3/2}) M  \Big( 1 + \eps \sup_{[0,t]}   \| \na_x\Phi  \|_{L^{\infty}_{x,v}}
\Big)
 $$
 and by a continuity argument, we obtain     
 \begin{equation}\label{ddx}
 \sup_{[0,t]}   \| \na_x\Phi  \|_{L^{\infty}_{x,v}} \lesssim \delta    (1+t^{3/2})M,
   \end{equation}
as long as $\delta  (1+t^{5/2}) \le \lambda_1$, for some small $\lambda_1>0$. 
For what concerns the $v$ derivative, we have
$$ \begin{aligned}
\| \na_v\eta (t) \|_{L^{ \infty}_{x,v}} &\lesssim \int_0^t  \Big( (1+ \| \na_v\eta (s) \|_{L^{\infty}_{x,v}} ) \| \na_x\Phi (s) \|_{L^{\infty}_{x,v}}  
\\&\quad + \delta  (1+ \eps \| \na_v\eta(s)\|_{L^{\infty}_{x,v}})\| (E,B)(s) \|_{L^{ \infty}}  \Big)\, ds.\end{aligned}$$
and thus, using~\eqref{ddx}, as long as $t \leq \min (T_\eps, \lambda_1 \delta^{-2/5})$, we have
$$
  \sup_{[0,t]}   \| \na_v \eta \|_{L^{\infty}_{x,v}} \lesssim \delta (1+t^{5/2})M.
$$
The $W^{2,\infty}$ bound for $\Phi-v$ can then be obtained exactly as for the estimate for $\na_x \Phi$,  noticing that $\pa^\alpha_x \pa^\beta_v (\Phi-v) = \pa^\alpha_x \pa^\beta_v \Phi$, as soon as $|\alpha| + |\beta| \geq2$.
This yields \eqref{estim-Burgers-infini}.

For what concerns \eqref{estim-Burgers-infini-bis}, we use a Taylor expansion to write
$$
\hat{\Phi} - \hat{v} = \int_0^1 (\Phi-v) \cdot \na \chi (s \Phi + (1-s) v) ds,
$$
where $\chi(u)= u/\sqrt{ 1 + \eps^2 |u|^2}$ and thus
$$
\pa_i \chi(u) = \frac{e_i}{\sqrt{1+ \eps^2 |u|^2} }- u \frac{\eps^2 u_i}{(1+ \eps^2 |u|^2)^{3/2}}, 
$$
where $e_i$ {is the i-th vector of the canonical basis of} $\R^3$.
Therefore, using \eqref{estim-Burgers-infini}, we deduce
$$
\sup_{[0,t]} \| ( 1 + \eps^2 |v|^2)^{1 \over 2}(\hat{\Phi} - \hat{v}) \|_{W^{2,\infty}_{x,v}} \lesssim \delta  (1+ t^{5/2}) M.
$$
The lemma is finally proved.
\end{proof}

  We now introduce the characteristics associated to $\Phi$, defined as the solution to    
   \begin{equation}
    \label{characteristic}
    \left\{
\begin{aligned}
     &\partial_{t} X(t,s,x,v)=  \hat{\Phi}(t,X(t,s,x,v), {v}), \\ 
     &X(s,s,x,v)= x,
     \end{aligned}
     \right.
    \end{equation}
    and study the deviation of $X$ from the (relativistic) free transport flow, following \cite[Lemma 5.2]{HKR}.
    
          \begin{lemma}
   \label{straight}
There is  $0<\lambda_2 \leq \lambda_1$ such that the following holds.
   For every $0 \leq s , t \leq T \leq \min (T_\e, \lambda_2\delta^{-2/7})$, we may write
    \begin{equation}
    \label{estimXtilde}
     X(t,s,x,v)=  x  + (t-s) \left( \hat{v} +   \tilde X(t,s,x,v)  \right)
     \end{equation}
     with $\tilde X$ that satisfies the estimate
      \begin{equation}
      \label{LinftyXtilde}
 \begin{aligned}
    \sup_{t,s \in [0,T]} \| \tilde{X}(t,s,x,v) \|_ {W^{2,\infty}_{x, v}} 
   \lesssim  \delta (1+T^{5/2})M.
   \end{aligned}
   \end{equation}
    
   Moreover, 
the map $x\mapsto x+ (t-s) \tilde X(t,s,x,v)$ is a diffeomorphism,  
and
there exists  $\Psi(t,s, x,  v)$ such that the identity
   \begin{equation}
   \label{identitePsi} X(t,s, x , \Psi(t,s,x, v)) = x +  (t-s) \hat{v}
   \end{equation}
   holds. Finally, we have the  estimate
          \begin{equation}
 \begin{aligned}
     \label{estimPsi}
     \sup_{t,s \in [0,T]} \|   \Psi(t,s,x,v)  - v \|_ {W^{2,\infty}_{x, v}}
        \lesssim  \delta (1+T^{5/2})M.
  \end{aligned}
  \end{equation}

   \end{lemma}

%
%
\begin{proof}[Proof of Lemma~\ref{straight}]

Set 
    $Y(t,s,x,v)= X(t,s,x,v)  - x- (t-s)\hat{v}$.     
    Note that we have
        \begin{equation}
    \label{eqintY}
    Y(t,s,x,v)= \int_{s}^t (\hat{\Phi} - \hat{v})\left(\tau, x+ (\tau -s ) \hat{v} + Y(\tau, s, x, v), v \right) \, d \tau,
    \end{equation}
    we deduce from Lemma \ref{lem-phi} that for $|\alpha | \leq 2$, we have for $0 \leq s, \, t  \le T \le \min(T_\e, \lambda_1 \delta^{-2/5})$,
    \begin{equation*} \sup_{|\alpha| \leq 2 } \| ( 1 + \eps^2 |v|^2)^{1 \over 2} \partial^\alpha_{x,v} Y(t,s) \|_{L^\infty_{x,v}} 
     \leq  \int_{s}^t  \delta (1+T^{5/2})M
     \big( 1 +  \sup_{|\alpha|  \leq 2 } \| ( 1 + \eps^2 |v|^2)^{1 \over 2} \partial^\alpha_{x,v} Y(\tau,s) \|_{L^\infty_{x,v}}\big)\, d\tau.
      \end{equation*}
This yields, for $0 \leq s, \, t  \le  T \le \ \min(T_\e,\lambda_1 \delta^{-2/7})$,
      \begin{equation}
      \label{estY1}  \sup_{|\alpha| \leq 2 } \| ( 1 + \eps^2 |v|^2)^{1 \over 2} \partial^\alpha_{x,v}Y(t,s) \|_{L^\infty_{x,v}} \lesssim |t-s| \delta (1+T^{5/2})M.
      \end{equation}
       Finally, we  set $\tilde{X}(t,s,x,v)=Y(t,s,x,v)/(t-s)$  and deduce
        that $\tilde X$ satisfies \eqref{LinftyXtilde}.  
       
       From~\eqref{estY1} we also deduce that $x\mapsto x+ (t-s) \tilde X(t,s,x,v)$ is a diffeomorphism. 
       Finally, let us prove that we can choose $\Psi$ such that \eqref{identitePsi} is verified.
      From \eqref{estY1}, we first observe that the map $v \mapsto  v + ( 1 + \eps^2 |v|^2)^{1 \over 2} \tilde{X}(t,s,x,v)$
       is a diffeomorphism. This allows to choose $\Psi_{1}(t,s,x,v)$ that satisfies \eqref{estimPsi} such that
       $$  \hat{\Psi}_1(t,s,x,v)+ \tilde{X}(t,s,x, \Psi_{1}(t,s,x,v)) = { v \over  (1 + \eps^2 |\Psi_{1}(t,s,x,v)|^2)^{1 \over 2}}:= H(t,s,x,v).$$
        Next, we want to prove that we can find $\Psi_{2}(t,s,x,v)$ such that 
        $ H(t,s,x, \Psi_{2}(t,s,x,v)) = \hat v$. Again, we observe that for every $v\in \mathbb{R}^3$, there exists a unique $w
        = \Psi_{2}(t,s,x,v)$
         in $\mathbb{R}^3$ such that
         $$   { w \over  (1 + \eps^2 |\Psi_{1}(t,s,x,w)|^2)^{1 \over 2}} = \hat v.$$
          Indeed, this amounts to prove that there exists a unique fixed point for the map 
          $$w \mapsto  \hat v \sqrt{ 1 + \eps^2 
            |\Psi_{1}(t,s,x,w)|^2}.$$  
            By using that $\Psi_{1}$ satisfies \eqref{estimPsi}, we get 
           that   $$\left| \nabla_{w}  \left( \hat v \sqrt{ 1 + \eps^2 
            |\Psi_{1}(t,s,x,w)|^2}\right) \right| \lesssim  \eps | \hat v|   \delta (1+T^{5/2})M \lesssim    \delta (1+T^{5/2})M$$
            and thus that  $w \mapsto  \hat v \sqrt{ 1 + \eps^2 
            |\Psi_{1}(t,s,x,w)|^2}$ is a contraction for every $v$ for $\lambda_{2}$ small enough. The estimates for $\Psi_{2}$ follow easily.
             We finally get \eqref{identitePsi} by setting $\Psi(t,s,x,v)= \Psi_{1}(t,s,x,\Psi_{2}(t,s,x,v)).$

\end{proof}

\subsection{Averaging operators}
\label{sec-av}

For  a smooth vector field $U(t,s,x,v)$,  we define  the following integral operator $ K_{U}$ acting on functions $H(t,x)$:
 $$ K_{U}(H) (t,x) =    \int_{0}^t \int_{\R^3} (\nabla_{x}  H) (s,  x - (t-s) \hat{v}) \cdot
      U(t,s,x,v)\, dv ds.$$
The integral operator $K$  can be seen as relativistic version of the operator $\mathsf{K}$
 $$ \mathsf{K}_{U}(H) (t,x) =    \int_{0}^t \int_{\R^3} (\nabla_{x}  H) (s,  x - (t-s) {v}) \cdot
      U(t,s,x,v)\, dv ds$$
      that was studied in \cite{HKR}.
We have the following proposition which is a consequence of  \cite[Proposition 5.1 and Remark 5.1]{HKR}. 
\begin{proposition}
\label{propK}
Let $k>4$ and $\sigma>3/2$. There holds, 
for all $H \in L^2([0, T]; L^2_{x})$,
     $$ \| K_{U}(H)\|_{L^2([0, T]; L^2_{x})} \lesssim \sup_{0 \leq s,\, t \leq T} \| U(t,s,\cdot) \|_{H^k_{2k+\sigma+1}} \|H\|_{L^2([0, T]; L^2_{x})}.$$
\end{proposition}

\begin{remark}
{\em In \cite{HKR},  this kind of estimate 
is used as a way to overcome the apparent loss of derivative in $x$ in the expression of the operator $\mathsf{K}$.
In the context of this work, Proposition~\ref{propK} will be useful to gain powers in time; indeed a use of the Cauchy-Schwarz inequality  yields an additional factor $t$. However we shall not apply it systematically as it is a bit costly in terms of regularity.
}\end{remark}

\begin{proof}[Proof of Proposition~\ref{propK}]
The idea is to come down from the relativistic to the classical operator by using the change of variables $p:= \hat{v}$. We have
\begin{align*}
K_{U}(H) &(t,x) \\
&= \int_{0}^t \int_{B(0, 1/\eps)} (\nabla_{x}  H) (s,  x - (t-s) p) \cdot
      U(t,s,x,v(p)) \frac{1}{(1-\eps^2|p|^2)^{5/2}}\, dp ds \\
      &=  \mathsf{K}_{ \mathsf{U} } (H)(t,x).
\end{align*}
with $v(p) := \frac{p}{\sqrt{1-\eps^2|p|^2}}$ and 
$$\mathsf{U}(t,s,x,p) := U(t,s,x,  v(p))  \frac{{1}_{p \in B(0, 1/\eps)}}{(1-\eps^2|p|^2)^{5/2}}.$$
  Let $\sigma > 3/2$. By \cite[Proposition 5.1 and Remark 5.1]{HKR}, we get
$$
 \| K_{U}(H)\|_{L^2([0, T]; L^2_{x})} =  \|  \mathsf{K}_{ \mathsf{U} } (H)\|_{L^2([0, T]; L^2_{x})}
 \lesssim \sup_{0 \leq s,\, t \leq T} \|  \mathsf{U}(t,s,\cdot) \|_{H^k_{\sigma}} \|H\|_{L^2([0, T]; L^2_{x})}.
$$
Observing that $|\pa^\alpha_p v(p) | \lesssim \frac{1}{(1-\eps^2|p|^2)^{1/2+ |\alpha|}}$, we have
$$
\|  \mathsf{U}(t,s,\cdot) \|_{H^k_{\sigma}} \lesssim \| {U}(t,s,\cdot) \|_{H^k_{2k+\sigma+1}},
$$
hence the claimed estimate.
\end{proof}

\subsection{The closed intregro-differential equation for the density}
\label{dens}
The  key algebraic step is the obtention of an integro-differential equation for $\rho$ that can be established from \eqref{main2}, using the averaging operators introduced in Section~\ref{sec-av} and the estimates of Lemmas~\ref{lem-phi} and \ref{straight}.

Let us first proceed with the first reduction in solving \eqref{main1}. 

\begin{lem}
\label{reduc1} For all $t \leq \min (T_\e, \lambda_2 \delta^{-2/7} )$, there holds the  identity
     \begin{equation}
     \label{eqrho}\rho(G) = K_{ \mathcal{U} }( (- \Delta)^{-1} \rho(G)) + \mathcal{S}_0 + \mathcal{S}_1,
     \end{equation}
in which 
     \begin{equation}
     \begin{aligned}
       \mathcal{U}(t,s,x,v) &:=    \nabla_{v}\mu(\Phi(s,x-(t-s)\hat{v},\Psi(s,t,x,v)))
     \\ &\quad \times   |\det \na_v \Phi(t,x,\Psi(s,t,x,v))|  | \det \na_v \Phi(t,x,v)|,
     \end{aligned}
     \end{equation}
     and
     \begin{equation}\label{def-S01}
\begin{aligned}
\mathcal{S}_0 &:=  \int_{\R^3} G_{|t=0} (X(0,t,x,v),v)|\det \na_v \Phi(t,x,v)| \, dv, \\
\mathcal{S}_1 &:=  \int_0^t \int_{\R^3} \mathcal{R}' (s, X(s,t,x,v),v) |\det \na_v \Phi(t,x,v)| \, dv  ds ,
\end{aligned}
\end{equation} 
and the remainder $\mathcal{R}' (t,x,v) := \mathcal{R}(t,x, \Phi(t,x,v))$ is defined as in 
\eqref{main2}. 
\end{lem}

\begin{proof}[Proof of Lemma~\ref{reduc1}]
We start from \eqref{main2}, which we recall reads
\begin{equation}\label{main2-re}
\partial_t F + \hat \Phi(t,x,v) \cdot \nabla_x F   - \na_x \phi[G] \cdot \na_v \mu (\Phi(t,x,v)) =  \mathcal{R}',
\end{equation}
with $F (t,x,v) = G (t,x, \Phi(t,x,v))$.  We integrate along the characteristics (recall the definition of $X$ in \eqref{characteristic}) to obtain
$$
\begin{aligned}
F(t,x,v) &= \int_0^t    \na_x \phi [G](s, X(s,t,x,v)) \cdot \na_v \mu (\Phi(t,X(s,t,x,v),v))  \, ds  \\
&+   \int_0^t    \mathcal{R}' (s, X(s,t,x,v),v) \, ds + F_{|t=0} (X(0,t,x,v),v) .
\end{aligned}
$$
We then multiply by the Jacobian $ |\det \na_v \Phi(t,x,v)|$ and integrate with respect to the velocity variable.
Upon making a change of variables and introducing $\mathcal{S}_0,\mathcal{S}_1$ as in \eqref{def-S01}, we obtain
$$
\begin{aligned}\rho(G) &= \int_{\R^3} G(t,x,v) \, dv \\
&=  \int_{\R^3} F(t,x,v)  |\det \na_v \Phi(t,x,v)| \, dv \\
&= \mathcal{S}_0  + \mathcal{S}_1
\\&+  \int_0^t  \int_{\R^3}    \na_x \phi [G](s, X(s,t,x,v)) \cdot \na_v \mu (\Phi(t,X(s,t,x,v),v))  |\det \na_v \Phi(t,x,v)| \, dv ds.
\end{aligned}$$
Using the change of variables $v \mapsto \Psi(t,s,x,v)$ introduced in Lemma~\ref{straight}, we have
\begin{align*}
&\int_0^t  \int_{\R^3}    \na_x \phi[G] (s, X(s,t,x,v)) \cdot \na_v \mu (\Phi(t,X(s,t,x,v),v))  |\det \na_v \Phi(t,x,v)| \, dv ds \\
&= \int_0^t  \int_{\R^3}  \na_x \phi [G](s, x - (t-s)\hat{v})   \cdot \na_v \mu (\Phi(t, x - (t-s)\hat{v},\Psi(s,t,x,v))) \\
&\qquad \qquad \times |\det \na_v \Phi(t,x,\Psi(s,t,x,v))|  | \det \na_v \Psi(t,x,v)| \, dv ds.
\end{align*}
This ends the proof of the lemma. 
\end{proof}

We further simplify the equation on $\rho(G)$ with the next reduction.

\begin{lem}
\label{reduc2}
For all $t \leq \min (T_\e, \lambda_2 \delta^{-2/7})$, the identity \eqref{eqrho} reduces to
     \begin{equation}
     \label{eqrho2}\rho(G) = K_{ \na_v \mu }( (- \Delta)^{-1} \rho(G))   + \mathcal{S}_0 + \mathcal{S}_1 + \mathcal{R}_1 
     \end{equation}
     with $  \mathcal{S}_0 + \mathcal{S}_1 $ defined as in Lemma \ref{reduc1} and the remainder $\Rc_1$ satisfying the estimate
          \begin{equation}
          \label{estim-reduc2}
   \| \Rc_1 \|_{L^2(0,T; L^2_x)} \lesssim  \delta (1+T^{7/2}) M \| \rho(G)\|_{L^2(0,T; L^2_x)}
     \end{equation} 
for all $T \leq \min (T_\e, \lambda_2 \delta^{-2/7})$.

\end{lem}

\begin{proof}[Proof of Lemma~\ref{reduc2}]

We write
\begin{align*}
       \mathcal{U}(s,t,x,v) &=  \na_v  \mu(v) \\
       &+ (  \nabla_{v}\mu(\Phi(s,x-(t-s)\hat{v},\Psi(s,t,x,v))) - \na_v \mu(v))
     \\
     &+  \nabla_{v}\mu(\Phi(s,x-(t-s)\hat{v},\Psi(s,t,x,v)))   \big(|\det \na_v \Psi(t,x,v)| -1\big), 
     \\
     &+   \nabla_{v}\mu(\Phi(s,x-(t-s)\hat{v},\Psi(s,t,x,v))) |\det \na_v \Psi(t,x,v)| \\
     &\qquad \qquad \times \big(| \det \na_v \Phi(t,x,\Psi(s,t,x,v))|-1\big) 
     \\
     &=:  \na_v  \mu(v) + \mathfrak{U}(s,t,x,v).
\end{align*}
We therefore set accordingly 
$
\Rc_1 :=  K_\mathfrak{U}[ (- \Delta)^{-1} \rho[G]].
$
Let us estimate the $L^2(0,T; L^2_x)$ norm of $\Rc_1$. We have
$$
\begin{aligned}
\| \Rc_1(t) \|_{L^2_x}^2
&\lesssim t  \int_0^t  \Big[\int_{\R^3} \Big\| |\na_x \phi [G]|(s,x-(t-s)\hat{v})   (1+|v|^2)\mathfrak{U}(s,t,x,v) \Big\|_{ L^2_x}^2 \, dv \\
&\qquad  \qquad \times \int_{\R^3} \, \frac{dv}{(1+|v|^2)^2} \Big]ds \\ 
&\lesssim t  \int_0^t   \| \na_x \phi [G]|(s)\|^2_{L^2_x}  \int_{\R^3}  (1+|v|^2)^2 \| \mathfrak{U}(s,t,x,v)\|^2_{L^\infty_x}  \, dv ds \\
&\lesssim t   \| \na_x \phi [G]|\|^2_{L^2(0,T; L^2_x)}  \sup_{0\leq s,t \leq T}  \int_{\R^3}  (1+|v|^2)^2 \| \mathfrak{U}(s,t,x,v)\|^2_{L^\infty_x}  \, dv.
\end{aligned}
$$ 
By the estimates of Lemmas  \ref{lem-phi} and \ref{straight}, we have
\begin{align*}
&\sup_{0\leq s,t \leq T}    \| \nabla_{v}\Phi(s,x-(t-s)\hat{v},\Psi(s,t,x,v))) - v) \|_{L^\infty_{x,v}} \lesssim \delta (1+T^{5/2})M, \\
&\sup_{0\leq s,t \leq T}   \||\det \na_v \Psi(t,x,v)| -1\|_{L^\infty_{x,v}} \lesssim \delta (1+T^{5/2})M, \\
&\sup_{0\leq s,t \leq T}   \| |\det \na_v \Phi(t,x,\Psi(s,t,x,v))| -1 \|_{L^\infty_{x,v}} \lesssim \delta (1+T^{5/2})M, \end{align*}
for  $T\leq  \min (T_\e, \lambda_2\delta^{-2/7})$.
Therefore, using a Taylor expansion and using the fast decay of $\mu$ at infinity, we obtain
 $$\left(\sup_{0\leq s,t \leq T}  \int_{\R^3}  (1+|v|^2)^2 \| \mathfrak{U}(s,t,x,v)\|^2_{L^\infty_x}  \, dv\right)^{1/2}
 \lesssim \delta (1+T^{5/2})M.
 $$
We deduce
$$
   \| \Rc_1 \|_{L^2(0,T; L^2_x)} \lesssim  \delta (1+T) (1+T^{5/2})M \| \rho(G)\|_{L^2(0,T; L^2_x)},
$$
which corresponds to \eqref{estim-reduc2}. 
\end{proof}


\subsection{Penrose inversion} The final step consists in an inversion of the integro-differential equation using the Penrose stability condition.
We end up with: 

\begin{lem}
\label{final}
For all $T \leq \min (T_\e, \lambda_2 \delta^{-2/7}) $, we have
\begin{equation}
\label{L2}
\| \rho(G) \|_{L^2(0,T; L^2_x)} \lesssim  \| \mathcal{S}_0 + \mathcal{S}_1 + \mathcal{R}_1 \|_{L^2(0,T; L^2_x)}.
\end{equation}
\end{lem}

\begin{proof}[Proof of Lemma~\ref{final}]
Using the change of variables $p:= \hat{v}$, we note that we can write
$$
K_{\na_v \mu} =  \mathsf{K}_{\mathsf{U}},
$$
with
$$
\mathsf{U}(p):= (\na_v \mu)(v(p)) \frac{1_{p \in B(0, 1/\eps)}}{(1-\eps^2|p|^2)^{5/2}}, \qquad  v(p) = \frac{p}{\sqrt{1-\eps^2|p|^2}}.
$$
According to  \cite{MV}, one can invert the operator 
$Id - \mathsf{K}_{\mathsf{U}}$ in $L^2_t L^2_x$ 
if
\begin{equation}
\label{newPen}
\inf_{\gamma>0, \tau \in \R, k \in \Z^3} \left| 1 + \int_0^{+\infty} e^{-(\gamma+i\tau)s} \widehat{\mathsf{U}} ( ks) \cdot \frac{k}{|k|^2} s \, ds \right| >0.
\end{equation}
We note, using the smoothness and the fast decay of $\mu$ at infinity, that
$$
\sup_{\gamma>0, \tau \in \R, k \in \Z^3} \left| \int_0^{+\infty} e^{-(\gamma+i\tau)s} \left(\widehat{\mathsf{U}} ( ks) - \widehat{\na_v \mu} (ks)\right) \cdot \frac{k}{|k|^2} s \, ds \right| 
\leq C_\mu  \eps^2.
$$
Since $\mu$ is stable in the sense of Penrose (recall \eqref{Pen}), this implies that for $\eps$ small enough, \eqref{newPen} is verified.
We end up with estimate~\eqref{L2}.
\end{proof}

Using~\eqref{estim-reduc2}, we can absorb the contribution of $\mathcal{R}_1$ into the left hand side of \eqref{L2}. Indeed, for $t \le \min(T_\eps, \lambda_3\delta^{-2/7})$,  
with a small $\lambda_3$, the estimate \eqref{L2} reduces to 
\begin{equation}\label{new-rhoS01}
\| \rho(G) \|_{L^2(0,T; L^2_x)} \lesssim  \| \mathcal{S}_0\|_{L^2(0,T; L^2_x)} +\|  \mathcal{S}_1  \|_{L^2(0,T; L^2_x)},
\end{equation}
with $\mathcal{S}_j$ being defined in \eqref{def-S01}. It remains to estimate these norms. This will be done in  the next two lemmas.


\begin{lem}
\label{lem-R} Let $\mathcal{S}_1$ be defined as in \eqref{def-S01}, and set $\lambda= \min(\lambda_0, \lambda_2,\lambda_3)$. 
For all $T\leq \min(T_\eps, \lambda \delta^{-2/(2n+3)})$, we have the bound
\begin{equation}
\left\| \mathcal{S}_1 \right\|_{L^2(0,T;L^2_x)} 
\lesssim \delta (1+T^{n+5/2}) M^2 + \e (1 + T^2)M ^2 .
\end{equation}

\end{lem}

\begin{proof}[Proof of Lemma~\ref{lem-R}]
By Cauchy-Schwarz (recall $k>2$) and the estimate  \eqref{estim-Burgers-infini}, we first  get the bound
$$
\begin{aligned}
\left\| \mathcal{S}_1 \right\|_{L^2(0,T;L^2_x)}  &=\left\|  \int_0^t \int_{\R^3} \mathcal{R}' (s, X(s,t,x,v),v) |\det \na_v \Phi(t,x,v)| \, dv  ds  \right\|_{L^2(0,T;L^2_x)}  \\
&\lesssim\left\| t^{1/2} \Big( \int_0^t \iint |\mathcal{R}' |^2 (s, X(s,t,x,v),v) (1+ |v|^2)^k   \, dv  dx ds \Big)^{1/2} \right\|_{L^2(0,T)}. 
\end{aligned}
$$
The strategy consists in applying successive changes of variables.
We start with $v \mapsto \Psi(t,s,x,v)$, $x \mapsto x+ (t-s)\hat v$ and use estimate~ \eqref{estimPsi}; this procedure yields 
$$
\begin{aligned}
&\left\| t^{1/2}  \Big(\int_0^t \iint |\mathcal{R}' |^2 (s, x-(t-s)\hat{v} ,\Psi(s,t,x,v)) (1+ |\Psi(s,t,x,v)|^2)^k   \, dv  dx ds\Big)^{1/2}  \right\|_{L^2(0,T)}  \\
&\lesssim\left\| t^{1/2} \Big( \int_0^t \iint |\mathcal{R}' |^2 (s, x ,\Psi(s,t,x+(t-s)\hat{v},v))   (1+ |v|^2)^k  \, dv  dx ds\Big)^{1/2}  \right\|_{L^2(0,T)}  .
\end{aligned}
$$
We then  apply $v \mapsto \Psi^{-1} (s,t,x+ (t-s) \hat{v}, v)$, $v \mapsto \Phi(t,x,v)$ and  use at multiple times the estimates \eqref{estim-Burgers-infini} and \eqref{estimPsi}; we end up with a bound by%
$$
\begin{aligned}
&\left\| t^{1/2} \Big(\int_0^t \iint |\mathcal{R}' |^2 (s, x ,v)  (1+ |v|^2)^k   |\det \na_v \Psi(t,x+(t-s)\hat{v},v)|^{-1} \, dv  dx ds\Big)^{1/2}  \right\|_{L^2(0,T)}  \\
&\lesssim\left\| t^{1/2} \Big(\int_0^t \iint |\mathcal{R} |^2 (s, x ,\Phi(t,x,v))  (1+ |v|^2)^k   \, dv  dx ds \Big)^{1/2} \right\|_{L^2(0,T)}  \\
&\lesssim\left\| t^{1/2}  \Big(\int_0^t \iint |\mathcal{R} |^2 (s, x ,v)  (1+ |v|^2)^k    \, dv  dx ds\Big)^{1/2}  \right\|_{L^2(0,T)}  \\
&\lesssim T \| \mathcal{R} \|_{L^2(0,T; L^2_k)},
\end{aligned}
$$
in which we recall $ \mathcal{R}' (t,x,v) := \mathcal{R}(t,x, \Phi(t,x,v)).$ As $\mathcal{R}$ satisfies the estimate~\eqref{R1'}, the lemma follows. 
%
%
\end{proof}

%

\begin{lem}
\label{lem-init}
Let $\mathcal{S}_0$ be defined as in \eqref{def-S01}, with $G_{\vert_{t=0}} = \pa^\alpha_x g_{\vert_{t=0}}$ for $|\alpha|\le n$. For all $T\leq \min(T_\eps, \lambda \delta^{-2/(2n+3)}, \lambda \eps^{-2})$, we have the bound
$$
 \left\| \mathcal{S}_0\right\|_{L^2(0,T;L^2_x)} \lesssim \| g_{\vert_{t=0}} \|_{H^n_k},
$$
with $k>13/2$.
\end{lem}

\begin{proof}[Proof of Lemma~\ref{lem-init}] 
We first eliminate the case of  small times, for which we have, arguing as in Lemma~\ref{lem-R},
$$
 \left\| \int_{\R^3} \pa^\alpha_x g_{|t=0} (X(0,t,x,v),v)|\det \na_v \Phi(t,x,v)| \, dv\right\|_{L^2(0,\tau_0;L^2_x)} \lesssim \| g |_{t=0} \|_{H^n_k}.
$$
 Then, by the change of variables $v \mapsto \Psi(0,t,x,v)$, we have
\begin{align*}
 &\quad \int_{\R^3} \pa^\alpha_x g_{|t=0} (X(0,t,x,v),v)|\det \na_v \Phi(t,x,v)| \, dv  \\
 &=\int_{\R^3} \pa^\alpha_x g_{|t=0} (x-t \hat{v}, \Psi(0,t,x,v))|\det \na_v \Phi(t,x,\Psi(0,t,x,v))| | \det \na_v \Psi(0,t,v)| \, dv
 \end{align*}
and we write the identity (using the notation $\hat{v} = (\hat{v}_1, \hat{v}_2,\hat{v}_3)$ and  $\pa^\alpha_x = \pa_{x_k} \pa^{\alpha'}_x$):
$$
\begin{aligned}
\pa^\alpha_x g_{|t=0} (x-t \hat{v}, \Psi(0,t,x,v)) &= \pa_{x_k} \pa^{\alpha'}_x g_{|t=0} (x-t \hat{v}, \Psi(0,t,x,v)) \\
&= - \frac{1}{t \pa_{v_k} \hat{v}_k }\Bigg(\pa_{v_k} \big[\pa^{\alpha'}_x g_{|t=0} (x-t \hat{v}, \Psi(0,t,x,v))\big] \\
&\quad + \pa_{v_k} \Psi(0,t,x,v) \cdot \big[ \na_v \pa^{\alpha'}_x g_{|t=0}\big] (x-t \hat{v}, \Psi(0,t,x,v)) \Bigg)  \\
&\quad + \sum_{j \neq k}  \frac{ \pa_{v_k} \hat{v}_j}{\pa_{v_k} \hat{v}_k}  \pa_{x_j}  \pa^{\alpha'}_x g_{|t=0} (x-t \hat{v}, \Psi(0,t,x,v)) \\
&=: I_1 +I_2 +I_3.
\end{aligned}
$$
Let us consider the contribution of $I_3$, for which we observe that
$$
 \pa_{v_k} \hat{v}_k  = \frac{1+ \eps^2 \sum_{j\neq k} v_j^2 }{(1+ \eps^2 |v|^2)^{3/2}}, 
 \qquad
 \pa_{v_k} \hat{v}_j  = -\frac{\eps^2 v_k v_j}{(1+ \eps^2 |v|^2)^{3/2}},  
$$
so that, we have for $j \neq k$, for all $v \in \R^3$,
$$
\left| \pa_{v_k} \hat{v}_j \right| \lesssim \eps |v|, \qquad  \left| \frac{1}{\pa_{v_k} \hat{v}_k }\right| \lesssim (1+ \eps^2|v|^2)^{3/2}
$$
and arguing again as in Lemma~\ref{lem-R}, we get
$$
  \left\| \int_{\R^3} I_3 |\det \na_v \Phi(t,x,v)| \, dv\right\|_{L^2(\tau_0,t;L^2_x)} \lesssim \eps \sqrt{t} \| g _{|t=0} \|_{H^n_k} 
$$
for $t\leq \min(T_\eps, \lambda \delta^{-2/(2n+3)})$, in which $k>13/2$ is used. This term accounts for the limitation $T \lesssim \eps^{-2}$.
For what concerns the contribution of $I_2$, we can directly use the extra $1/t$ factor to integrate in time and obtain with  a similar strategy as in the previous lemmas
$$
  \left\| \int_{\R^3} I_2 |\det \na_v \Phi(t,x,v)| \, dv\right\|_{L^2(\tau_0,t;L^2_x)} \lesssim  \left( \int_{\tau_0}^{+\infty} \frac{ds}{s^2}\right)^{1/2}  \| g _{|t=0} \|_{H^n_k} \lesssim  \| g _{|t=0} \|_{H^n_k} .
$$
For $I_1$ we use an integration by parts (in $v$) argument. We get
$$
 \int_{\R^3} I_1 |\det \na_v \Phi(t,x,v)| \, dv = \frac{1}{t}  \int_{\R^3} \pa^{\alpha'}_x g_{|t=0} (x-t \hat{v}, \Psi(0,t,x,v)) \eta(t,x,v) \, dv,
$$
with 
$$
\eta(t,x,v) := \pa_{v_k} \left( \frac{ | \det \na_v \Psi(0,t,v)|}{ \pa_{v_k} \hat{v}_k } \right) 
$$
It follows from  a straightforward computation that 
$$
\left | \pa_{v_k} \left( \frac{1}{ \pa_{v_k} \hat{v}_k } \right)  \right| \lesssim \eps^4 |v|^2,
$$
so that using also Lemma~\ref{straight}, for all $v$, there holds
$$
\sup_{t,x} |\eta(t,x,v) | \lesssim (|v|^3+ \eps^4 |v|^2).
$$
Therefore, using the estimates~\ref{estim-Burgers-infini} and the extra $1/t$ factor, we deduce
$$
  \left\| \int_{\R^3} I_1 |\det \na_v \Phi(t,x,v)| \, dv\right\|_{L^2(\tau_0,t;L^2_x)} \lesssim   \| g _{|t=0} \|_{H^n_k}.
$$
The lemma is finally proved.
\end{proof}


\subsection{Conclusion}
\label{sec-con} Gathering all pieces together, we are ready to deduce the following  $L^2(0,t; H^n_x)$ estimates for the densities $\rho(g)$ and $j(g)$.

\begin{lem}
\label{final-rj}
For $n\geq 4$, there is a small $\lambda>0$ such that the following holds:
\begin{equation}
\label{rj-final}
\| \rho(g) ,j(g) \|_{L^2(0,t; H^n_x)} \lesssim   \|   g_{|t=0}  \|_{H^n_k} +  \e (1 + t^{2})M ^2 + \delta (1+t^{n+5/2}) M^2
\end{equation}
for any $t \leq \min (T_\e, \lambda \delta^{-2/(2n+3)}, \lambda \eps^{-2})$. 
\end{lem}
\begin{proof}[Proof of Lemma~\ref{final-rj}] We start with the $L^2(0,t; H^n_x)$ estimate for $\rho$. The estimate for $j$ will be considered afterwards.

According to Lemma~\ref{commu}, for all $|\alpha|=n$, $ \pa^\alpha_x g$ satisfies an equation of the form~\eqref{main1}. Therefore \eqref{new-rhoS01} holds for $G =  \pa^\alpha_x g$. Using  Lemmas \ref{lem-R} and \ref{lem-init}, we obtain at once  
\begin{equation*}
\begin{aligned}
\| \rho( \pa^\alpha_x g) \|_{L^2(0,t; L^2_x)} &\lesssim   \| \mathcal{S}_0 \|_{L^2(0,t; L^2_x)}  + \| \mathcal{S}_1  \|_{L^2(0,t; L^2_x)} 
\\
&\lesssim   
 \|   g_{|t=0}  \|_{H^n_k} + \delta (1+t^{n+5/2}) M^2 + \e (1 + t^2)M^2.
\end{aligned}
\end{equation*}
Since $\rho(g)$ has zero average (recall the normalization  \eqref{avg-initial}), the above yields 
\begin{equation}
\label{Linfty-rho}
\| \rho( g) \|_{L^2(0,t; H^n_x)} \lesssim  \|   g_{|t=0}  \|_{H^n_k} + \e (1 + t^{2})M ^2 + \delta (1+t^{n+5/2}) M^2.
\end{equation}

We now  estimate  $j(g)$. We argue as in the beginning as the proof of Lemma~\ref{reduc1}, except for the fact that we now integrate against the weight $\hat v$. 
We end up with 
\begin{align*}
j(\pa^\alpha_x g) =  \int_0^t  \int_{\R^3} \na_x \phi[\pa^\alpha_x g] (s, x-(t-s)\hat{v}) \cdot \na_v \mu (v)  
  \hat v \, dv ds + \mathfrak{S},
 \end{align*}
 where $\mathfrak{S}$ is a remainder satisfying
 $$
 \| \mathfrak{S} \|_{L^2(0,t; L^2_x)} \lesssim    \delta (1+t^{7/2}) M^2 + t\| {R}_\alpha \|_{L^2(0,t; L^2_k)}  +  \| g_{|t=0}  \|_{H^n_k} ,
 $$ 
upon mimicking Lemmas~\ref{reduc2}, \ref{lem-R} and~\ref{lem-init}, and with $k>15/2$. 
Recalling the bound~\eqref{R1} on $R_\alpha$, the remainder $\mathfrak{S}$ satisfies the same bound as that of $ \rho(\partial_x^\alpha g)$.

Next, we use Proposition~\ref{propK} to bound the main contribution in the expression of $j(\pa^\alpha_x g)$ as follows. There holds
$$
\begin{aligned}
\Big
\| \int_0^\tau & \int_{\R^3} \na_x \phi[\pa^\alpha_x g] (s, x-(\tau-s)\hat{v}) \cdot \na_v \mu (v)  \hat v \, dv ds \Big\|_{L^2(0,t; L^2_x)} \\
&\lesssim C_\mu \| \rho(\partial_x^\alpha g)\|_{L^2(0,t; L^2_x)} \lesssim  \| \rho(\partial_x^\alpha g)\|_{L^2(0,t; L^2_x)} . 
\end{aligned}$$
That is, we have proved
\begin{equation}
\label{j1}
\| j(\pa^\alpha_x g) \|_{L^2(0,t; L^2_x)} \lesssim   \|   g_{|t=0}  \|_{H^n_k} +  \e (1 + t^{2})M^2
+ \delta (1+{t}^{n+5/2})M^2.
\end{equation}
In order to estimate the $H^n_x$ norm of $j(g)$, there remains to compute the contribution of the average of $j(g)$.

\begin{lem}
\label{av-j}
For any $t \leq \min (T_\e, \lambda \delta^{-2/(2n+3)})$, we have
\begin{equation}
\label{jdx}
\left\|  \int j(g) \, dx  \right\|_{L^2(0,t)} \lesssim  \delta (1+{t}^{n+2})M^2 + \delta \e (1 + t^{n+3})M^2.
\end{equation}
\end{lem}

\begin{proof}[Proof of Lemma~\ref{av-j}]
From the Vlasov equation~\eqref{main} satisfied by $g$, we have
$$
\begin{aligned}
\frac{d}{dt} \int j(g) \, dx &= \int R \hat v \, dv dx + \delta \begin{pmatrix} \int E_1 (\int (\pa_{v_1} \hat{v}) g \, dv) \, dx \\ \int E_2 (\int (\pa_{v_2} \hat{v}) g \, dv) \, dx \\ \int E_3 (\int (\pa_{v_3} \hat{v}) g \, dv) \, dx \end{pmatrix}
+ \delta \eps \int j(g) \times B \, dx 
\end{aligned}$$
which we rewrite, after integrating in time, as
$$
 \int j(g) \, dx =: J_1 + J_2 + J_3 +  \int j(g_{|t=0}) \, dx.
$$
For $J_1$, we apply Cauchy-Schwarz and estimate~\eqref{R1} to get
$$
\begin{aligned}
J_1 &\lesssim t^{1/2} \| R\|_{L^2(0,t; L^2_k)} \lesssim  \delta (1+t^{n+2}) M^2 
  \end{aligned}
$$
Note here that we actually used the fact that the contribution of $ -\eps \hat{v} \cdot \na_x (A \cdot \na_v \mu) $ in $R$ vanishes thanks to the integration in $x$. This observation will be useful later for the improvement in terms of $\eps$ for higher regularity; see Section \ref{sec-high}.

For $J_2$, using $|\pa_{v_i} \hat{v} |\lesssim 1$, we have
$$
\begin{aligned}
 J_2  &\lesssim \delta t^{1/2}  \| E \|_{L^2(0,t; H^n_x)} \| g \|_{L^\infty(0,t;L^2_k)} 
 \\&\lesssim \delta (1+t^{n+2}) M^2  + \delta\eps (1+t^{n+3}) M^2.
\end{aligned}
$$
according to Lemma~\ref{prem}.

For $J_3$, we argue exactly as for $J_2$, which yields
$$
\begin{aligned}
 J_3  &\lesssim \delta \eps  t^{1/2}   \| B \|_{L^2(0,t; H^n_x)} \| g \|_{L^\infty(0,t;L^2_k)} 
 \\&\lesssim \delta\eps (1+t^{n+2}) M^2  + \delta\eps^2 (1+t^{n+3}) M^2.
\end{aligned}
$$
Finally for the contribution of the initial condition, we have, recalling  the definitions~\eqref{j-g} and~\eqref{norm-A} and the fact that the initial condition is normalized so that $\int f_0 \hat{v} \, dv dx =0$,
$$
\begin{aligned}
 \int j(g_{|t=0}) \, dx  &=  \underbrace{\int j(f_0) \, dx}_{=0} +  \e \lambda(\mu,\e) \underbrace{\int A_{|t=0}  \,dx}_{=0} =0,   \end{aligned}
$$
and we deduce~\eqref{jdx}.
\end{proof}

Applying Lemma~\ref{av-j} for $t \leq 1/\eps$, we deduce
$$
\left\|  \int j(g) \, dx  \right\|_{L^2(0,t)} \lesssim  \delta (1+{t}^{n+2})M^2.
$$
Gathering all pieces together, we obtain the claimed bound on $j(g)$. Lemma~\ref{final-rj} is finally proved.
 \end{proof}






We are finally in position to close the bootstrap argument. Indeed by Lemmas~\ref{prem} and \ref{final-rj}, there is $C_0>0$ depending only on $\mu$, such that,
for all $t \leq \min (T_\e, \lambda \delta^{-2/(2n+3)}, \lambda \eps^{-2})$,
$$
\Nc(t) \leq C_0 \left(\|f_0 \|_{H^n_k} + \eps \|B_0 \|_{H^n_x}  + \e (1 + t^{2})M ^2 + \delta (1+t^{n+5/2}) M^2\right).
$$
Now, choose $M$ large enough so that
$$
C_0 \|f_0 \|_{H^n_k}  < \frac{1}{2} M.
$$
There are $\lambda>0$ and $\eps_0, \delta_0>0$ (all small enough) such that for all $\delta \leq \delta_0$, $\eps \leq \eps_0$,  
$$
 C_0(2\eps  + \lambda^2)M^2
+  C_0 (\delta+\lambda^{n+5/2})M^2 < \frac{1}{2} M.
$$
As a consequence if we had $T_\eps \leq  \min (\lambda \eps^{-1/2}, \lambda \delta^{-2/(2n+5)})$, we would have
$$
\Nc(T_\eps) < M,
$$
which cannot be by definition of $T_\eps$. This therefore implies that
\begin{equation}
T_\eps   \geq  \min (\lambda \eps^{-1/2}, \lambda \delta^{-2/(2n+5)})
\end{equation}
and the proof of Proposition~\ref{key-prop} is complete.
%

\section{Proof of Theorem \ref{thm} }\label{Sec-6}
Starting from Proposition \ref{key-prop}, the proof of our main theorem is now straightforward. Indeed, applying Lemmas~ \ref{EB} and \ref{prem}, for all $t \le\min (\lambda \eps^{-1/2}, \lambda \delta^{-2/(2n+5)})$, we estimate
$$ \| (E^\eps(t), B^\eps(t)) \|_{L^2(0,t; H^n_x)} \lesssim \delta (1+t) M $$ 
and 
$$
\begin{aligned}
\|f(t) \|_{H^n_k} &\lesssim \|g(t) \|_{H^n_k}  +  \| \eps A \|_{H^n_x} \\
&\lesssim (1+t^{n+1/2}) M  + \eps (1+t^{n+3/2}) M
\end{aligned}
$$
so that
$$
\begin{aligned}
 \| f^\eps(t) - \mu \|_{H^n_k}  = \delta\| f(t)\|_{H^n_k} 
 &\lesssim \delta  (1+t^{n+1/2}) M  + \delta \eps (1+t^{n+3/2}) M
 \\&\lesssim \delta  (1+t^{n+1/2}) M 
 .\end{aligned}$$ 
Theorem \ref{thm} follows.

\section{Proof of Theorem \ref{thm} (case $n\geq 6$)}\label{sec-high}

For the case $n \geq 6$, we apply as well the same bootstrap argument, based on Sobolev norms of order $n$, with the aim to prove the  improved proposition:
\begin{prop}
\label{key-prop2}
Let $n \geq 6$. Assume all requirements of the statement of Theorem~\ref{thm}.
There is $M>0$ so that the following holds. Define
$$
T_\eps = \sup\{T\geq 0, \,  \| (\rho(g), j(g) \|_{L^2(0,T; H^n_x)} \leq M\}.
$$
There are $\eps_0, \delta_0>0$ and $\lambda>0$, such that for all $\eps\leq \eps_0, \delta\leq \delta_0$,
$$
T_\eps \geq \lambda  \min (  \eps^{-1},  \delta^{-2/(2n+5)}),
  $$
  where $T_\eps$ is defined in \eqref{def-T}.
\end{prop}
Note the improvement is precisely for what concerns the order in $\eps$, compared to Proposition~\ref{key-prop}. The beginning of the argument remains the same. The only significant change appears for Lemma~\ref{commu} in which we do not bound directly the contribution of the remainder term 
$$\mathfrak{R} :=-\eps \hat v \cdot \na_x (A \cdot \na_v \mu).$$
We will treat it in a specific way in the computation of the densities by applying Proposition~\ref{propK}; this is how we will improve from times of order  $\eps^{-1/2}$ to order $\eps^{-1}$ in the case $n\geq 6$.

In order to use Proposition~\ref{propK}, we will need a slightly more precise version of the estimates of
Lemmas~\ref{lem-phi} and~\ref{straight}. To obtain these bounds, one may also exactly proceed like we did before to prove the aforementioned lemmas.

\begin{lem}
\label{lelem}
With the same notations as in Lemmas~\ref{lem-phi} and~\ref{straight}, we have
for all 
$T \leq \min(T_\e, \lambda_1  \delta^{-2/7})$
\begin{align*}
  \sup_{t \in [0,T] } \| \Phi  - v\|_{W^{r,\infty}_{x,v}}  &\lesssim  \delta (1+t^{5/2})M, \qquad \forall r < n - 3/2, \\
     \sup_{t  \in [0,T]  } \sup_v\| \Phi  - v\|_{H^{r}_{x}}  &\lesssim  \delta (1+t^{5/2})M, \qquad \forall r \leq  n,
\end{align*}
and
\begin{align*}
  \sup_{s,t \in [0,T] } \| \Psi  - v\|_{W^{r,\infty}_{x,v}}  &\lesssim  \delta (1+t^{5/2})M, \qquad \forall r < n - 3/2, \\
     \sup_{s,t \in [0,T] } \sup_v \| \Psi  - v\|_{H^{r}_{x}}  &\lesssim  \delta (1+t^{5/2})M, \qquad \forall r \leq  n.
\end{align*}
\end{lem}

The following of the bootstrap argument remains unchanged, except for the key improvement which corresponds to
\begin{lem}
\label{impro}For all $T\leq \min(T_\eps, \lambda \delta^{-2/7})$ we have the bound
$$
\left\|  \int_0^t \int_{\R^3} \mathfrak{R}(s, X(s,t,x,v),\Phi(t,x,v)) |\det \na_v \Phi(t,x,v)| \, dv  ds  \right\|_{L^2(0,T;L^2_x)} \lesssim \eps( 1+T)M.
$$
\end{lem}
Recall that in the previous proof, we could only obtain a bound by $ \eps( 1+T^2)M$ (see Lemma~\ref{lem-R}) which accounted for the limitation  of order $\eps^{-1/2}$.

\begin{proof}[Proof of Lemma~\ref{impro}]
We use the change of variables $v \mapsto \Psi(s,t,x,v)$ to write
\begin{align*}
& \int_0^t \int_{\R^3} \mathfrak{R}(s, X(s,t,x,v),\Phi(t,x,v)) |\det \na_v \Phi(t,x,v)| \, dv  ds \\
  &=\int_0^t \int_{\R^3} \mathfrak{R}(s, x -(t-s) \hat v,\Phi(t,x,\Psi(s,t,x,v))) \\
  &\qquad \qquad \qquad \qquad \qquad \qquad \times |\det \na_v \Phi(t,x,\Psi(s,t,x,v))|  |\det \na_v \Psi(s,t,x,v) | \, dv  ds \\
  &= \eps \int_0^t \int_{\R^3} \sum_j \na_x A_j (s, x -(t-s) \hat v) \cdot \tilde{\mu}_j(\Phi(t,x,\Psi(s,t,x,v))) \\
  & \qquad \qquad \qquad \qquad \qquad \qquad \times |\det \na_v \Phi(t,x,\Psi(s,t,x,v))|   |\det \na_v \Psi(s,t,x,v) | \, dv  ds \\
  &= \eps \sum_{j=1}^3  K_{\mathfrak{U}_j} (\na_x A_j),
\end{align*}
setting 
\begin{align*}
\tilde \mu_j (v) &:=  -\hat{v} \pa_{v_j} \mu(v), \\
\mathfrak{U}_j (s,t,x,v) &:=  \tilde{\mu}_j(\Phi(t,x,\Psi(s,t,x,v)))|\det \na_v \Phi(t,x,\Psi(s,t,x,v))| |\det \na_v \Psi(s,t,x,v) |  .
\end{align*}
Finally, we can use Sobolev bounds, the smoothness and fast decay of $\mu$,  Proposition~\ref{propK} and Lemma~\ref{lelem}  to get the estimate 
\begin{align*}
&\left\|  \int_0^t \int_{\R^3} \mathfrak{R}'(s, X(s,t,x,v),\Phi(t,x,v)) |\det \na_v \Phi(t,x,v)| \, dv  ds  \right\|_{L^2(0,T;L^2_x)} \\
&\lesssim \sum_j  \| \mathfrak{U}_j \|_{H^n_{2n+3}} \| \na_x A\| _{L^2(0,T;L^2_x)} \\
&\lesssim \eps( 1+T)M,
\end{align*}
which ends the proof of the lemma. 
\end{proof}
Consequently, this leads to the following improved form of Lemma~\ref{final-rj}:
\begin{lem}
\label{final-rj-impro}
For $n>\frac72$, there is a small $\lambda>0$ such that the following holds:
\begin{equation}
\label{rj-final-impro}
\| \rho(g) ,j(g) \|_{L^2(0,t; H^n_x)} \lesssim   \|   g_{|t=0}  \|_{H^n_k} +  \e (1 + t)M ^2 + \delta (1+t^{n+5/2}) M^2
\end{equation}
for any $t \leq \min (T_\e, \lambda \delta^{-2/(2n+3)}, \lambda \eps^{-2})$. 
\end{lem}

Concluding as before, we end up with the proof of Theorem~\ref{thm} in the case $n\geq 6$.

\section{Beyond Poisson: the Darwin approximation (Proof of Theorem~\ref{thm2})}
\label{sec-D}
 
To go beyond the scale $1/\eps$, we develop 
a high order linearized Vlasov-Darwin approximation of the linearized Vlasov-Maxwell system.

To implement this idea, we also need to modify the bootstrap norm. 
Let us first define the high order moments
$$
m_\ell (g)  := \Bigg(\int \prod_{j=1}^\ell \hat v_{i_j} g  \, dv\Bigg)_{i_1, \cdots, i_\ell \in \{1,2,3\}}, \qquad \ell \geq 2,
$$
and introduce a notation for space averages. For any function or vector field $\psi(x)$, we set
$$
\langle \psi \rangle := \int_{\T^3} \psi (x) \,dx.
$$
Let $N\geq 2$. We introduce the new bootstrap norm $\mathfrak{N}$ (compare to the previous one in \eqref{def-N})
\begin{equation}
\begin{aligned}
\label{newdef-N}
\mathfrak{N}(t)
 := &\left\|(\rho(g),j(g))\right \|_{L^2(0,t; H^n_x)} +\sum_{\ell=2}^{2N+1}\left\| m_\ell (g) - \langle m_\ell (g)\rangle  \right\|_{L^2(0,t; H^n_x)}.
\end{aligned}
\end{equation}
We shall prove
\begin{prop}
\label{key-prop3}
Assume all requirements of the statement of Theorem~\ref{thm2}.
There is $M>0$ so that the following holds. Define
$$
T_\eps = \sup\{T\geq 0, \,  \mathfrak{N}(t) \leq M\}.
$$
There are $\eps_0, \delta_0>0$ and $\lambda>0$, such that for all $\eps\leq \eps_0, \delta\leq \delta_0$,
$$
T_\eps \geq \lambda  \min (  \eps^{-\min(N+1/2,p)},  \delta^{-1/(n+3)}).
  $$
\end{prop}

\subsection{Darwin approximation}
The Darwin approximation (see e.g. \cite{BK}) consists
in introducing the vector potential  $A_1$ with average $  \langle A_1 \rangle  =0$, solving
\begin{equation} \label{def-A11}
(-\Delta + \eps^2 \lambda) A_1  = \eps \P\widehat j (g) , \qquad 
\widehat j (g) :=  j (g) -   \langle j(g) \rangle.
\end{equation}
where $\P$ denotes the Leray projector on divergence free vector fields, and $\lambda = \lambda(\mu,\eps)$ is defined as in \eqref{j-g}.   
By construction, $A_1$ satisfies the Coulomb gauge $\na \cdot A_1=0$. For convenience, we shall set in the following 
$$ -\Delta_\eps := -\Delta + \eps^2 \lambda.$$
We set 
$$\mathfrak{A}_1 := \widehat A - A_1, \qquad \widehat A := A -  \langle A \rangle.$$ 
By construction, $\mathfrak{A}_1$ satisfies the wave equation
\begin{equation} 
\label{wawa0}
\eps^2 \pa_t^2 \mathfrak{A}_1 -\Delta_\eps \mathfrak{A}_1   = - \eps^2 \pa_t^2 A_1, \qquad \langle \mathfrak{A}_1 \rangle =0,
\end{equation}
upon noting that $\eps^2 (\pa_t^2 + \lambda)\langle A \rangle = \eps \langle j(g)\rangle$ and $\P j(g) = j(g) - \pa_t \nabla \phi$. 

Let us start by analyzing the source term $\eps^2 \pa_t^2 A_1$ in the wave equation \eqref{wawa0}. 

\begin{lem}\label{newA1-dt} For $t \le \min (T_\eps, \lambda_0 \delta^{-2/(2n+3)})$, there holds
\begin{equation} 
\| \eps^2 \pa_t^2 A_1 -  \eps^3 M_1(f)  \|_{L^1(0,t; H^n_x)}  \lesssim  \eps^2  \delta (1 + t^{n+2})M^2 
\end{equation}
in which 
$$M_1(f) :=  \P (-\Delta_\eps)^{-1}  \int \hat v (\hat v\cdot \nabla_x)^2 f  \; dv .$$
\end{lem} 
\begin{proof}[Proof of Lemma~\ref{newA1-dt}] Recalling \eqref{def-A11}, we first compute $\partial_t^2 \P \widehat j(g)$. From the Vlasov equation \eqref{clVM-pert}, we have
\begin{align*}
\pa_t j(f) &= - \int \hat{v} (\hat{v}\cdot \nabla_x) f\,dv  + \lambda(\eps, \mu)  E+ \delta \int_{\R^3} F \cdot \na_v     \hat{v}  f  \, dv
\end{align*}
with $F= E + \varepsilon \hat{v} \times B$.
We note that $\P E = - \eps \pa_t A$. Recalling that $j(g) = j(f) +  \eps \lambda (\eps, \mu)A$, the above yields 
\begin{equation}\label{dtPgg}
\pa_t \P j(g) = - \P\int \hat{v} (\hat{v}\cdot \nabla_x) f\,dv  +  \delta  \P \int_{\R^3} F \cdot \na_v     \hat{v}   f  \, dv.
\end{equation}
Therefore, we obtain
$$
\begin{aligned}
\pa_t^2 \P j(g)  
&=  - \P\int \hat{v} (\hat{v} \cdot \nabla_x)\pa_tf\,dv
+ \delta \P \int_{\R^3} F \cdot \na_v     \hat{v}   \pa_t f   \, dv + \delta \P \int_{\R^3} \pa_tF \cdot \na_v     \hat{v}  f  \, dv
\\& =: J_1 + J_2 + J_3 .
\end{aligned}
$$
Using again the Vlasov equation  we compute  
$$
\begin{aligned}
J_1 
& =  \P \int \hat v (\hat v\cdot \nabla_x)^2 f \; dv  + \P \nabla_x \cdot\int \hat{v} \otimes\hat{v} \nabla_v \cdot (\mu E) \,dv 
\\&\quad - \delta \P \nabla_x\cdot \int F \cdot \na_v    (\hat{v} \otimes \hat{v})  f \,dv.
\end{aligned}
$$
Since $\mu \equiv \mu(|v|^2)$, we remark that
$$
\int \hat{v} \otimes\hat{v} \pa_{v_j} \mu  \,dv=0, \quad \forall j=1,2,3,
$$
and therefore we have
$$
\begin{aligned}
J_1 & =  \P \int \hat v (\hat v\cdot \nabla_x)^2 f \; dv  - \delta \P \nabla_x\cdot \int  F \cdot \na_v    (\hat{v} \otimes \hat{v}) f\,dv
\\&=: J_{11} + J_{12}.
\end{aligned}$$
Note that the terms $J_{12}, J_2,J_3$ have a factor of $\delta$ in their expression. For $t \le \min (T_\eps, \lambda_0 \delta^{-2/(2n+3)})$, 
these terms can be handled using the crude weighted Sobolev bounds of Lemmas~\ref{EB} and~\ref{prem}. Namely, we have 
$$\begin{aligned} 
\| J_{12} - \langle J_{12}\rangle\|_{L^1 (0,t; H^{n-2})} &\lesssim \delta \| (E,B)\|_{L^1(0,t; H^{n-1})} \| f \|_{L^\infty(0,t; H_k^{n-1})}
\\&\lesssim  \delta (1+t^{3/2}) \Big[  (1+t^{n-1/2})   + \eps (1+t^{n+1/2}) \Big] M^2 
\\
&\lesssim \delta (1 + t^{n+2})M^2 .
\end{aligned}$$
Similarly, using the Vlasov equation, we have
$$
\begin{aligned}
J_2 & =  - \delta \P \int_{\R^3}   F \cdot \na_v     \hat{v}  f   \, dv - \delta \P \int_{\R^3} F \cdot \na_v     \hat{v} \nabla_v \cdot (\mu E)   \, dv 
\\&\quad + \delta^2 \P \int_{\R^3} \sum_{j,k=1}^3 (\pa^2_{v_j, v_k}   \hat{v} )  F_j F_k f   \, dv.
\end{aligned}$$
Consequently, using Lemmas~\ref{EB} and~\ref{prem}, we estimate 
$$
\begin{aligned}
\|J_2 - \langle &J_2\rangle\|_{L^1(0,t;H^{n-2})} \\
& \lesssim  \delta \| (E,B)\|_{L^1(0,t; H^{n-2})} \| f \|_{L^\infty(0,t; H_k^{n-1})} + \delta \| (E,B)\|^2_{L^2(0,t; H^{n-2})} 
\\&\quad  + \delta^2 \| (E,B)\|^2_{L^2(0,t; H^{n-2})} \| f \|_{L^\infty(0,t; H_k^{n-2})}   
\\& \lesssim \delta (1+t^{3/2}) \Big[  (1+t^{n-1/2})   + \eps (1+t^{n+1/2}) \Big] M^2 + \delta (1+t^2) M
\\&\quad + \delta^2 (1+t^2) \Big[  (1+t^{n-3/2})   + \eps (1+t^{n-1/2}) \Big] M^2
\\&\lesssim   \delta (1 + t^{n+2})M^2 .
\end{aligned}$$
Next, using the Maxwell equations
\begin{align*}
\pa_t E = \frac{1}{\eps} \na_x \times B - j(f), \qquad \pa_t B = - \frac{1}{\eps} \na_x \times E,
\end{align*}
we compute 
$$ J_3= \delta \P  \int_{\R^3}    \Big[ \frac{1}{\eps} \na_x \times B - j(f)  - \hat v \times (\na_x \times E)\Big] \cdot \na_v   \hat{v}  f\; dv$$
and hence, using Lemmas~\ref{EB} and~\ref{prem}, we get
$$
\begin{aligned}
\|J_3 - \langle J_3\rangle\|_{L^1(0,t; H^{n-2})}
&\lesssim \delta \sqrt{t} \Big[ \frac{1}{\eps}\| E,B\|_{L^2(0,t; H^{n-1})} \| f \|_{L^\infty(0,t; H^{n-2}_k)} \\
&\quad + (\| j(f) \|_{L^2(0,t; H^{n-1})}+ \| E\|_{L^2(0,t; H^{n-1})} )\| f \|_{L^\infty(0,t; H^{n-2}_k)}
\Big]\\
&\lesssim \eps^{-1} \delta (1 + t^{n+2})M^2.
\end{aligned}
$$
Finally, we recall that 
$$ \eps^2 \pa_t^2 A_1 = \eps^3 (-\Delta_\eps)^{-1} \pa_t^2 \P \widehat j(g). $$ 
The above estimates on $J_{12}, J_2,J_3$ yield 
\begin{equation}\label{dt2-A11} \| \eps^2 \pa_t^2 A_1  - \eps^3  (-\Delta_\eps)^{-1}  J_{11}  \|_{L^1(0,t; H^{n})} \lesssim \eps^2  \delta (1 + t^{n+2})M^2 .\end{equation}
This proves the lemma with $M_1(f) = (-\Delta_\eps)^{-1} J_{11}$.
\end{proof}


Lemma \ref{newA1-dt} proves that the remainder $\mathfrak{A}_1 = \hat A - A_1$ solves  the wave equation
\begin{equation} 
\label{wawa11}
\eps^2 \pa_t^2 \mathfrak{A}_1 -\Delta_\eps \mathfrak{A}_1 =  - \eps^3 M_1(f)  + R,
\end{equation}
with $\| R \|_{L^1(0,t;H^n)} \lesssim \eps^2\delta (1 + t^{n+2})M^2$.
Recalling $f = g + \eps A \cdot \nabla_v \mu$, the remaining term on the right of the wave equation can be estimated, using 
\begin{equation}\label{M1-re}
\begin{aligned}
\eps^3  M_1(f) 
&= \eps^3 \P(-\Delta_\eps)^{-1}  \int \hat v (\hat v\cdot \nabla_x)^2 (g + \eps A \cdot\nabla_v \mu)  \; dv 
\end{aligned}\end{equation}
where the first term is bounded by $\eps^3 \sqrt t \mathfrak{N}(t)$ in $L^1(0,t;H^n)$ in view of the definition of $\mathfrak{N}(t)$. 
This term may thus reduce the times of validity (although stopping here would already actually allow to go beyond the scale $1/\eps$).
To reach arbitrary high order,  a \emph{higher-order Darwin approximation} is needed.


\subsection{Higher-order Darwin approximation} We start  again with the wave equation \eqref{wawa11}, having the leading remainder $\eps^3 M_1(f) $ as in \eqref{M1-re}. First, observe 
from~\eqref{M1-re} and the fact that $\mu \equiv \mu(|v|^2)$ that we can write 
$$ M_1(f) = M_1(g) + 2\eps \Big( (-\Delta_\eps)^{-1} \int \mu' (|v|^2) \frac{v \otimes v}{\sqrt{1+ \eps^2 |v|^2}} (\hat v\cdot \nabla_x)^{2}  \; dv\Big) A.$$
The operator applied to $A$ is symmetric and bounded in $L^2_x$. It is therefore natural to introduce another potential vector $A_2$, solving a similar elliptic problem to \eqref{def-A11}, to absorb the contribution of the moment $M_1(g)$. 

For $k\ge 1$, let us set 
\begin{equation}\label{def-Bkk}
M_k(f) :=   \P (-\Delta_\eps)^{-k} \int \hat v (\hat v\cdot \nabla_x)^{2k} f  \; dv 
\end{equation}
and 
\begin{equation}\label{def-Skk}  \mathcal{S}_k: = 2(-\Delta_\eps)^{-k}\int \mu' (|v|^2) \frac{v \otimes v}{\sqrt{1+ \eps^2 |v|^2}} (\hat v\cdot \nabla_x)^{2k}   \; dv.
\end{equation}
Clearly, $ \mathcal{S}_k$ is a symmetric and bounded operator on $L^2_x$. In addition, directly from $f = g + \eps A \cdot \nabla_v \mu$ and $\mu \equiv \mu(|v|^2)$, there holds the relation 
\begin{equation}\label{Mk-fg} 
M_k(f) = M_k(g) + \eps \mathcal{S}_k\widehat A , \qquad \forall ~k\ge 1.
\end{equation}
We note for later use that for  all distribution functions $h$ and any $s \geq 0$, 
\begin{equation}
\label{eq-use}
\| M_k(h) \|_{H^s_x} \lesssim \| m_{2k+1} (h) - \langle m_{2k+1} (h) \rangle  \|_{H^s_x}.
\end{equation}

We will define a sequence of high order approximation $(A_j)_{j\geq2}$ using these ``twisted'' high order moments.
We shall check eventually that the  remainder
\begin{equation}
\label{def-Phikk}
\mathfrak{A}_k := \widehat{A} - \sum_{j=1}^k A_j
\end{equation}
 is indeed better than the previous one in \eqref{wawa11}. 

The key algebraic lemma is the following statement.
\begin{lem}
\label{lem-Mk}
Let $M_k(f)$ and $\mathcal{S}_k$ be defined as in \eqref{def-Bkk} and \eqref{def-Skk}. For $t \le \min (T_\eps, \lambda_0 \delta^{-2/(2n+3)})$, there holds 
\begin{equation}
(-\Delta_\eps)^{-1}\pa_t^2 M_k(g) =  M_{k+1} (f) +  R_k,
\end{equation}
with $R_k$ satisfying the bound
$$
\| R_k \|_{L^1(0,t;H^n)} \lesssim \eps^{-1}\delta (1 + t^{n+2})M^2.
$$
\end{lem}

\begin{proof}[Proof of Lemma~\ref{lem-Mk}]
Using the Vlasov equation for $f$, we compute 
$$
\begin{aligned}
\partial_t \int \hat v (\hat v\cdot \nabla_x)^{2k} f  \; dv
&= \int \hat v (\hat v\cdot \nabla_x)^{2k} \partial_t f  \; dv 
\\
&= \int \hat v (\hat v\cdot \nabla_x)^{2k} \Big [ - \hat{v} \cdot \nabla_x f - E \cdot \nabla_v \mu - \delta\nabla_v \cdot (F f ) \Big] \; dv 
\end{aligned} $$
and hence 
$$
\begin{aligned}
\partial^2_t \int \hat v (\hat v\cdot \nabla_x)^{2k} f  \; dv
&=  \int \hat v (\hat v\cdot \nabla_x)^{2k} \Big [ - \hat{v} \cdot \nabla_x \partial_t f - \pa_t E \cdot \nabla_v \mu - \delta\nabla_v \cdot \pa_t (F f ) \Big] \; dv 
\\&= \int \hat v (\hat v\cdot \nabla_x)^{2k+2} f \; dv +  \int \hat v (\hat v\cdot \nabla_x)^{2k+1} E \cdot \nabla_v \mu \; dv  
\\&\quad -  \int \hat v (\hat v\cdot \nabla_x)^{2k} \pa_t E \cdot \nabla_v \mu
\\&\quad + \delta  \int \hat v (\hat v\cdot \nabla_x)^{2k} \Big[ (\hat v \cdot \nabla_x) \nabla_v \cdot (F f )- \nabla_v \cdot \pa_t (F f ) \Big] \; dv 
\\& =: 
I_0 + I_1 + I_2 + I_3.
\end{aligned} $$
By definition of $M_k(f)$ in~\eqref{def-Bkk}, we therefore have 
$$ \partial_t^2 M_k(f) = (-\Delta_\eps)^{-k} \P \Big( I_0 + I_1 + I_2 + I_3\Big).$$
By definition of $M_{k+1}(f)$, it follows that 
$$(-\Delta_\eps)^{-k} \P I_0 = -\Delta_\eps M_{k+1} (f).$$
Moreover, since $\mu \equiv \mu(|v|^2)$, we have
$$
\int \hat v (\hat v\cdot \nabla_x)^{2k+1} \pa_{v_j} \mu \; dv =0, \qquad \forall j=1,2,3,
$$
and thus
$$I_1 =0.$$ 
The term $(-\Delta_\eps)^{-k}I_3$ has a prefactor  $\delta$ in its expression and thus can be estimated in the same way as done for $J_{12},J_2,J_3$ in the previous lemma; we get
$$\|(-\Delta_\eps)^{-k}I_3 \|_{L^1(0,t; H^{n-2})} \lesssim \eps^{-1}\delta (1 + t^{n+2}).$$ 
As for $I_2$, using again $\mu \equiv \mu(|v|^2)$ and $\P E = - \eps \pa_t A$, we can write 
$$ I_2=  2\eps \P\int \mu' (|v|^2) \frac{v \otimes v}{\sqrt{1+ \eps^2 |v|^2}} (\hat v\cdot \nabla_x)^{2k} \pa_t^2A  \; dv .$$
By definition of $\mathcal{S}_k$ in \eqref{def-Skk}, we note 
$$  (-\Delta_\eps)^{-k} I_2 = \eps \mathcal{S}_k \pa_t ^2A =  \eps \mathcal{S}_k \pa_t ^2\widehat A.$$
The lemma follows, upon recalling \eqref{Mk-fg}. 
\end{proof}

We are ready to introduce the higher-order Darwin approximation, built with an inductive construction.

\begin{lem}
\label{lem-Ak}
 Let $M_k(f)$ and $\mathcal{S}_k$ be defined as in \eqref{def-Bkk} and \eqref{def-Skk}. For $t \le \min (T_\eps, \lambda_0 \delta^{-2/(2n+3)})$, the following holds. For each $k\ge 1$, there are symmetric and bounded operators $\{\cS_{k,j}\}_{j=1,\cdots,k}$ such that 
\begin{equation}\label{wave-kkk}
\begin{aligned}
\eps^2 \pa_t^2 \mathfrak{A}_k   - \Delta_{\eps,k} \mathfrak{A}_k  & =  \eps^{2k+1}\sum_{j=1}^k \cS_{k,j}M_j(f) + \mathfrak{R}_k,
\end{aligned}\end{equation}
where we recall $ \mathfrak{A}_k = \widehat A -\displaystyle{ \sum_{j=1}^k A_j}$ and $\mathfrak{R}_k$ satisfies the bound
$$
\|\mathfrak{R}_k \|_{L^1(0,t;H^n)} \lesssim \eps^2\delta (1 + t^{n+2})M^2.
$$
The vector fields $A_k$ and the operators $\Delta_{\eps,k}$ and $\cS_{k,j}$ are constructed inductively, following for all $k \geq 1$,
\begin{equation}\label{def-Ak123}\begin{aligned}
\Delta_{\eps,k+1}&:= \Delta_{\eps,k} + \eps^{2k+2}\sum_{j=1}^k \cS_{k,j} \cS_j,
\\
A_{k+1}&:=  \eps^{2k+1}(- \Delta_{\eps,k+1})^{-1} \sum_{j=1}^k \cS_{k,j} \Big[ M_j(g) +  \eps \cS_j\sum_{\ell =1}^k  A_\ell \Big],
\\
\cS_{k+1,j} &:=  -(- \Delta_{\eps,k+1})^{-1} \left[ (-\Delta_\eps) \cS_{k,j-1} +\sum_{i=1}^k \sum_{\ell = j}^k \eps^{2\ell}\cS_{k,i} \cS_{\ell,j} \cS_i \right], \\
&\qquad\qquad\qquad\qquad \forall j =1,\cdots,k+1,
\end{aligned}\end{equation}
with $\Delta_{\eps,1} := \Delta_\eps$, $\cS_{1,1} := - Id$, $\cS_{k,0}:=0$ for all $k \geq 1$, and $A_1$ is defined in \eqref{def-A11}. 
%
%
\end{lem}

\begin{proof}[Proof of Lemma~\ref{lem-Ak}] We start from the equation \eqref{wawa11}, which reads, using \eqref{Mk-fg},
$$
\eps^2 \pa_t^2\mathfrak{A}_1 - \Delta_\eps \mathfrak{A}_1 =  - \eps^3 M_1(f) + \mathcal{E}_0.
$$
Here and in what follows, the remainder $\mathcal{E}_0$ may change from line to line, but satisfies the uniform bound 
\begin{equation}\label{def-EEE}
\| \mathcal{E}_0\|_{L^1(0,t;H^n)}  \lesssim \eps^2\delta (1 + t^{n+2})M^2.
\end{equation}
Since  $M_1(f) = M_1 (g) + \eps \cS_1 A = M_{1}(g) +  \eps \cS_1 A_{1}  +  \eps \cS_1 \mathfrak{A}_1  $, 
 we   introduce 
$$A_2: = -\eps^3(-\Delta_{\eps,2})^{-1}  ( M_1(g) + \eps \cS_1 A_1) , \qquad \Delta_{\eps,2}: = \Delta_\eps - \eps^4 \cS_1,$$ 
and set $\mathfrak{A}_2= \mathfrak{A}_1 - A_2$. 
Note that since $\cS_1$ is bounded in $L^2$, for $\eps$ small enough, $-\Delta_{\eps,2}$ is invertible. It then follows that $\mathfrak{A}_2$ solves the wave equation
$$
\begin{aligned}
\eps^2 \pa_t^2 \mathfrak{A}_2   -\Delta_{\eps,2} \mathfrak{A}_2  &= -\eps^2 \pa_t^2 A_2+ \mathcal{E}_0.
\end{aligned}$$
From Lemma~\ref{lem-Mk} and Lemma~\ref{newA1-dt}, we have 
$$
\begin{aligned}
\eps^2\pa_t^2 A_2 & = -\eps^5(-\Delta_{\eps,2})^{-1} ( \pa_t^2M_1(g) + \eps \cS_1 \pa_t^2A_1) 
\\& = -\eps^5 (-\Delta_{\eps,2})^{-1} (-\Delta_\eps) M_2(f) - \eps^7 (-\Delta_{\eps,2})^{-1}  \cS_1M_1(f) + \eps^4 \cE_0.
\end{aligned} $$
Consequently, we obtain 
$$
\begin{aligned}
\eps^2 \pa_t^2 \mathfrak{A}_2   -\Delta_{\eps,2} \mathfrak{A}_2  &=  \eps^5 \left[\eps^2 (-\Delta_{\eps,2})^{-1}  \cS_1M_1(f)  +  (-\Delta_{\eps,1})^{-1} (-\Delta_\eps) M_2(f)  \right]+ \cE_0 \\
 &=  \eps^5 \sum_{j=1}^2 \cS_{k,j}M_j(f) + \cE_0,
\end{aligned}$$
setting
$$
\begin{aligned}
\cS_{2,1} &:= \eps^2 (-\Delta_{\eps,2})^{-1}  \cS_1, \\
\cS_{2,2} &:= (-\Delta_{\eps,1})^{-1} (-\Delta_\eps),
\end{aligned}$$
which proves the lemma for $k =2$.

By induction, we pick $k \geq 2$ and assume that $ \mathfrak{A}_k$ solves the wave equation 
$$
\begin{aligned}
\eps^2 \pa_t^2 \mathfrak{A}_k   -\Delta_{\eps,k} \mathfrak{A}_k  & =  \eps^{2k+1}\sum_{j=1}^k \cS_{k,j}M_j(f) + \cE_0
\end{aligned}$$
where the $\cS_{k,j}$ are symmetric and bounded operators and we assume that~\eqref{def-Ak123} is verified up to $k-1$.  
In particular $-\eps^2 \pa_t^2 A_k$ is equal to the source in the above wave equation, up to a term of the form $\cE_0$. To proceed further, it is natural to introduce the operator
$$
\Delta_{\eps,k+1}:= \Delta_{\eps,k} + \eps^{2k+2}\sum_{j=1}^k \cS_{k,j} \cS_j,
$$
that is invertible for $\eps$ small enough,
and define the next order approximation
$$A_{k+1} : =  \eps^{2k+1}(- \Delta_{\eps,k+1})^{-1} \sum_{j=1}^k \cS_{k,j} \Big[ M_j(g) +  \eps \cS_j\sum_{\ell =1}^k  A_\ell \Big].$$
Set $ \mathfrak{A}_{k+1} =  \mathfrak{A}_k - A_{k+1}$. It then follows that 
$$
\begin{aligned}
\eps^2 \pa_t^2 \mathfrak{A}_{k+1}   -\Delta_{\eps,k+1} \mathfrak{A}_{k+1} 
 &= - \eps^2 \pa_t^2 A_{k+1} + \cE_0,
 \end{aligned}$$
and by the induction assumption we compute,
$$ \begin{aligned}
\eps^2 \pa_t^2 A_{k+1} 
 &=  \eps^{2k+3}(- \Delta_{\eps,k+1})^{-1} \sum_{j=1}^k \cS_{k,j} \Big[ \pa_t^2M_j(g) +  \eps \cS_j  \sum_{\ell =1}^k\pa_t^2A_\ell \Big]  \\
 &=\eps^{2k+3}(- \Delta_{\eps,k+1})^{-1} \sum_{j=1}^k \cS_{k,j} \Big[ \pa_t^2M_j(g)  
 +   \cS_j \sum_{\ell =1}^k\eps^{2\ell} \sum_{i=1}^\ell \cS_{\ell,i}M_i(f)  \Big]   \\&\quad 
 +\eps^{2k+3} \cE_0.
 \end{aligned}$$
Therefore by Lemma~\ref{lem-Mk}, we have
$$ \begin{aligned}
\eps^2 \pa_t^2 A_{k+1}  &=  \eps^{2k+3}(- \Delta_{\eps,k+1})^{-1}(-\Delta_\eps) \sum_{j=1}^k \cS_{k,j}  \Big[  M_{j+1}(f) +   \cS_j \sum_{\ell =1}^k\eps^{2\ell} \sum_{i=1}^\ell \cS_{\ell,i}M_i(f)  \Big] 
 \\&\quad 
 +\eps^{2k+3} \cE_0
\\
 &=  -\eps^{2k+3}  \sum_{j=1}^{k+1} \cS_{k+1,j} M_j(f)  +\eps^{2k+3} \cE_0, \end{aligned}$$
 setting for $j=1,\cdots, k+1$,
 $$
 \cS_{k+1,j} :=  -(- \Delta_{\eps,k+1})^{-1} \left[ (-\Delta_\eps) \cS_{k,j-1} +\sum_{i=1}^k \sum_{\ell = j}^k \eps^{2\ell}\cS_{k,i} \cS_{\ell,j} \cS_i \right].
 $$
The lemma follows at once. 
\end{proof}

We finally justify in the following lemma that $\mathfrak{A}_k$ is indeed a high order remainder.

\begin{lemma}\label{newA} For $k\ge 2$, define $\mathfrak{A}_k$ as in \eqref{def-Phikk}. Then, for $t \le \min (T_\eps, \lambda_0 \delta^{-2/(2n+3)})$, we have
$$
\begin{aligned}
\sup_{[0,t]} 
 \| (\eps \mathfrak{A}_k, \eps \pa_t \mathfrak{A}_k, \nabla_x \mathfrak{A}_k) \|_{H^n}  
\lesssim  
R_0 +  \eps\delta (1 + t^{n+2})M^2  + \eps^{2k}(1+t^{3/2})M^2 
,
\end{aligned}
$$
in which $R_0: = \| (\eps\mathfrak{A}_k|_{t=0},\eps \pa_t \mathfrak{A}_k|_{t=0}, \na_x \mathfrak{A}_k |_{t=0}) \|_{H^n} $. 
\end{lemma}
\begin{proof} [Proof of Lemma~\ref{newA}]
Taking the $H^n$ scalar product with $\pa_t \mathfrak{A}_k$ in~\eqref{wave-kkk}, we obtain
$$
\begin{aligned}
\frac{d}{dt} \| (\eps \mathfrak{A}_k, \eps \pa_t \mathfrak{A}_k, (-\Delta_{\eps,k})^{1/2} \mathfrak{A}_k) \|_{H^n}^2  
&
\lesssim  \eps^{2k} \sum_{j=1}^k \| \cS_{k,j} M_j (f) \|_{H^n}  \| \eps \pa_t \mathfrak{A}_k \|_{H^n} 
\\&\quad + \frac{1}{\eps} \| \mathfrak{R}_k\|_{H^n} \| \eps\pa_t \mathfrak{A}_k\|_{H^n}  .
\end{aligned}
$$
We used the fact that by construction, $-\Delta_{\eps,k} = -\Delta + \eps^2 \mathcal{T}_k$ for some symmetric and bounded operator $\mathcal{T}_k$, and thus $-\Delta_{\eps,k}$ is a self-adjoint non-negative operator. Thus, integrating in time, we deduce
$$
\begin{aligned}
\sup_{[0,t]} 
 \| (\eps \mathfrak{A}_k, \eps \pa_t \mathfrak{A}_k, \nabla_x \mathfrak{A}_k) \|_{H^n}  \lesssim  R_0 +  \eps\delta (1 + t^{n+2})+ \eps^{2k} \sum_{j=1}^k \| M_j (f) \|_{L^1(0,t;H^n)},
\end{aligned}
$$
in which $R_0$ accounts for the initial data as written in the statement of the lemma. As for $M_j(f)$, we use the identity \eqref{Mk-fg}, Lemma~\ref{EB}, and the definition of $\mathfrak{N}(t)$ to conclude that 
$$\begin{aligned}
\| M_j (f) \|_{L^1(0,t;H^n)} 
&\lesssim \|  M_j (g) \|_{L^1(0,t;H^n)} + \eps\| \cS_j \widehat A\|_{L^1(0,t;H^n)}
\\
&\lesssim \sqrt t \mathfrak{N}(t) + (1+t^{3/2}) M \\
&\lesssim  (1+t^{3/2}) M,
\end{aligned}$$
for $t \le \min (T_\eps, \lambda_0 \delta^{-2/(2n+3)})$.
The lemma is finally proved. 
\end{proof}

\subsection{The well-prepared assumption}
\label{subsec-wp}

Now that the $A_j$ are properly defined, it is  time to define what we mean by well-prepared initial conditions. Take $k=N$.

\begin{definition}
\label{def-wp}We say that the initial condition $(f_0,E_0,B_0)$ is well-prepared of order $p$ if
\begin{equation}
\label{eq-wp}
\begin{aligned}
(-\Delta)^{-1} \na \times B_0 &= \sum_{j=1}^N A_j|_{t=0}+ \mathcal{O}(\eps^{p/2-1})_{H^{n+1}_x} \\
  \na_x (-\Delta)^{-1} \left(\int_{\R^3} f_0 \,dv\right)- E_0&= \sum_{j=1}^N \eps \pa_t A_j|_{t=0}+ \mathcal{O}(\eps^{p/2-1})_{H^n_x}.
\end{aligned}
\end{equation}
Recall that $A|_{t=0} = (-\Delta)^{-1} \na \times B_0$ and $ \eps \pa_t A|_{t=0} = - E_0 + \na_x (-\Delta)^{-1} \int_{\R^3} f_0 \, dv$.
\end{definition}

\begin{remark}
{\em Like in Remark~\ref{rem-j}, we can slightly generalize Definition~\ref{def-wp} by allowing in~\eqref{eq-wp} additional  errors of order $\mathcal{O}(\delta^b)$, with $b>1/18$.
}\end{remark}

From the very definition of $\mathfrak{A}_N$ in  \eqref{def-Phikk}, we therefore get 
\begin{lem}\label{newA-init}Assume that the initial condition is well-prepared. There holds
\begin{equation} 
\| (\e \mathfrak{A}_k, \e\partial_t \mathfrak{A}_k, \nabla_x \mathfrak{A}_k)|_{t=0}\|_{H^n_x}  \lesssim \eps^{p/2-1}.
\end{equation}
\end{lem}

It can be complicated in practice to check the well-prepared assumption for arbitrary values of $p$. However it is possible to write down explicitly the required conditions for $p=4,6,8$.
\begin{itemize}
\item For $p=4$, we note that 
$$
\| \sum_{j=1}^N A_j|_{t=0},\sum_{j=1}^N \eps \pa_t A_j|_{t=0} \|_{H^{n+1}_x\times H^n_x} \lesssim \eps,
$$
so that  \eqref{eq-wp} reads
\begin{equation}
\label{eq-4}
\| (A|_{t=0}, E_0 + \na_x \phi_0) \|_{H^{n+1}_x\times H^n_x}\lesssim \eps .
\end{equation}

\item For $p=6$, we note that 
$$
\| \sum_{j=2}^N A_j|_{t=0},\sum_{j=2}^N \eps \pa_t A_j|_{t=0} \|_{H^{n+1}_x\times H^n_x} \lesssim \eps^3,
$$
and using the Vlasov equation, that
$$
\| \eps^2 (-\Delta)^{-1}( \pa_t j(f))|_{t=0}\|_{H_x^n}\lesssim \eps^2,
 $$
so that  \eqref{eq-wp} reads
\begin{equation}
\label{eq-6}
\begin{aligned}
&\| A|_{t=0}- \eps (-\Delta)^{-1} j(f|_{t=0}) \|_{H_x^{n+1}}\lesssim \eps^2, \\
&\|  E_0 + \na_x \phi_0  \|_{H_x^n}\lesssim \eps^2.
\end{aligned} 
\end{equation}

\item For $p=8$,  \eqref{eq-wp} reads
\begin{equation}
\label{eq-8}
\begin{aligned}
&\| A|_{t=0}- \eps (-\Delta)^{-1} j(f|_{t=0}) \|_{H_x^{n+1}}\lesssim \eps^3, \\
&\|  (E_0 + \na_x \phi_0) -  \eps^2 (-\Delta)^{-1} (\pa_t j(f))|_{t=0}  \|_{H_x^n}\lesssim \eps^3,
\end{aligned} 
\end{equation}
where $ \pa_t j|_{t=0}$ can be computed using the trace of the Vlasov equation at time $0$.
\end{itemize}

%
%

\subsection{The closed equation on high order moments}

Take $k={N}$ and set 
$$\widetilde{A} := \mathfrak{A}_N$$
with $\mathfrak{A}_k$ being defined as in \eqref{def-Phikk}. We now start over the analysis of Sections~\ref{dens} to~\ref{sec-con}, except that we do not treat the contribution of the term 
\begin{equation}
\label{eq-nn}
-\eps \hat{v} \cdot \na_x \left(\sum_{\ell=1}^N A_\ell \cdot \na_v \mu\right)
\end{equation}
as a remainder. The main difference with the previous treatment 
comes from the fact that we have to study the system for $(\rho(g),j(g),m_\ell(g)-\langle m_\ell(g) \rangle)$ instead of $(\rho(g),j(g))$  alone.
Let us first introduce the new remainder term (compare with~\eqref{def-R})
\begin{equation}
\label{newdef-R}
\begin{aligned}
\widetilde{R}:= 
&-\eps \hat{v} \cdot \na_x (\A \cdot \na_v \mu) 
- \delta \eps (E+ \eps \hat{v} \times B) \cdot \na_v(A \cdot \na_v \mu).
\end{aligned}
\end{equation}
For $\ell=1$, according to~\eqref{def-A11}, we have
\begin{equation}
\label{eq-A1f}
A_1  = \eps (-\Delta_\eps)^{-1} \P\widehat j (g).
\end{equation}
For $\ell \geq 2$, according to~\eqref{def-Ak123}, we have
\begin{equation*}
\begin{aligned}
A_{\ell} &=   \eps^{2\ell-1}(- \Delta_{\eps,\ell})^{-1} \sum_{j=1}^{\ell-1} \cS_{\ell-1,j} \Big[ M_j(g) +  \eps \cS_j\sum_{r =1}^{\ell-1}  A_r \Big].
\end{aligned}
\end{equation*}
By a straightforward induction (recalling~\eqref{eq-use}), it follows that we can write $A_\ell$ as
\begin{equation}
\label{eq-nnn}
\begin{aligned}
A_\ell = \eps^{2\ell-1} (- \Delta_{\eps,\ell})^{-1}\left[ \mathsf{F}^\ell_1 (j(g)) + \sum_{r=2}^{2\ell-1} \mathsf{F}^\ell_r (m_r(g))\right],
\end{aligned}
\end{equation}
where the $\mathsf{F}^\ell_j$ are linear bounded operators on $L^2_{t,x}$.
We  consider the same notations as in Section~\ref{Sec-5}. 
\begin{equation}\label{main-final}
\begin{aligned}
\partial_t G &+ \hat{v} \cdot \nabla_x G + \delta (E + \varepsilon \hat{v} \times B)\cdot \nabla_v   G \\
  &- \na_x \phi[G] \cdot \na_v \mu + \eps \hat{v} \cdot \na_x \left(\sum_{\ell=1}^N A_\ell[G] \cdot \na_v \mu\right) = \widetilde{\mathcal{R}},
\end{aligned}
\end{equation}
with $\phi[G]$ solving 
$
-\Delta_x \phi[G] = \int_{\R^3} G \, dv
$
and the $A_\ell[G]$ are defined as in~\eqref{eq-A1f} and \eqref{eq-nnn}, with $G$ replacing $g$. The remainder $\widetilde{\mathcal{R}}$ has to be tought of as  $\pa^\alpha_x \tilde R$, for $|\alpha|=n$.

We follow the same approach as in Lemma \ref{reduc1} to obtain
\begin{lem}
\label{newreduc}
For all $t \leq \min (T_\e, \lambda_2 \delta^{-2/7})$, we have
     \begin{equation}
     \begin{aligned}
     \label{neweqrho-j}\rho(G) &= K_{ \na_v \mu }( (- \Delta)^{-1} \rho(G))   +  \eps^2 \left[\mathsf{L}^0_{1}(j(G)) + \sum_{r=2}^{2N+1}  \mathsf{L}^0_{r}(m_r(G))\right]
      + \mathsf{R}, \\
      j(G) &= \mathsf{L}^1_0 (\rho(G))   +   \eps^2 \left[\mathsf{L}^1_{1}(j(G)) + \sum_{r=2}^{2N+1}  \mathsf{L}^1_{r}(m_r(G))\right] + \mathsf{R}^1, \\
    m_\ell (G) &=  \mathsf{L}^\ell_{0} (\rho(G)) +  \eps^2 \left[\mathsf{L}^\ell_{1}(j(G)) + \sum_{r=2}^{2N+1}  \mathsf{L}^\ell_{r}(m_r(G))\right] + \mathsf{R}^\ell,    \quad  \ell =2, \cdots, 2N+1.
     \end{aligned}
     \end{equation}
     The operators $(\mathsf{L}^\ell_r)_{\ell,r}$ are bounded operators on $L^2_{t,x}$, with norm of order $1$, uniformly in $\eps$.
     The remainders are defined as 
   \begin{equation}\label{def-RRR}
     \begin{aligned}
      \mathsf{R} &:= \mathcal{S}_0 + \widetilde{\mathcal{S}_1} + \mathcal{R}_1, \\
\mathsf{R}^\ell  &:= \mathcal{S}_0^\ell + \widetilde{\mathcal{S}_1}^\ell + \mathcal{R}_1^\ell, \qquad \forall \ell = 1,\cdots, 2N+1,  
     \end{aligned}
\end{equation}
where
     \begin{itemize}
\item $ \mathcal{S}_0$ is defined as in Lemma~\ref{reduc1} and corresponds to the contribution of the initial conditions,
\item $\widetilde{\mathcal{S}_1} $  is defined as $\mathcal{S}_1$ in Lemma~\ref{reduc1}, except that we impose that $\widetilde{R}$ replaces $R$,
\item ${\Rc_1}$ is defined as in Lemma~\ref{reduc2} and  satisfies the estimate   
      \begin{equation}
          \label{newestim-reduc2}
\|{\Rc_1}\|_{L^2(0,t; L^2_x)} \lesssim  \delta (1+t^{7/2}) M^2,
     \end{equation} 
\item  $\mathcal{S}_0^\ell$, $\widetilde{\mathcal{S}_1}^\ell $, $ \mathcal{R}_1^\ell$ are defined in the same way as   $\mathcal{S}_0$, $\widetilde{\mathcal{S}_1} $, $ \mathcal{R}_1$, with  additional multiplications by $\hat{v}_{i_1} \cdots \hat{v}_{i_\ell} $ in the integrals, and satisfy similar estimates.
     \end{itemize}
     \end{lem}
 
 \begin{proof}[Proof of Lemma~\ref{newreduc}]
 We proceed exactly as in Lemma~\ref{reduc1}, except that we do not consider the contribution of the term \eqref{eq-nn}  as a remainder. The operators $\mathsf{L}^\ell_{0}$ are similar to
 $K_{\na_v \mu}$ except that they include multiplications by $\hat{v}_{i_1} \cdots \hat{v}_{i_\ell}$.
  Likewise, using  \eqref{eq-nnn} and recalling the definition of the $M_j$ in~\eqref{def-Bkk}, the operators $\mathsf{L}^\ell_{j}$ can be defined similarly to $K_{\na_v \mu}$. 
 \end{proof}

We then note that the map
$$
\begin{pmatrix} \rho(G)\\   j(G) \\ m_2(G) \\ \vdots \\ m_{2N+1}(G) \end{pmatrix} \mapsto \begin{pmatrix} \rho(G)- K_{ \na_v \mu }( (- \Delta)^{-1} \rho(G))   -  \eps^2 \left[\mathsf{L}^0_{1}(j(G)) + \displaystyle{\sum_{r=2}^{2N+1}}  \mathsf{L}^0_{r}(m_r(G))\right] \\ 
j(G)-  \mathsf{L}^1_{0} (\rho(G))   -  \eps^2\left[\mathsf{L}^1_{1}(j(G)) +\displaystyle{ \sum_{r=2}^{2N+1}}  \mathsf{L}^1_{r}(m_r(G))\right] \\
 m_2 (G) -  \mathsf{L}^2_{0} (\rho(G)) -
        \eps^2 \left[\mathsf{L}^2_{1}(j(G)) +\displaystyle{ \sum_{r=2}^{2N+1}}  \mathsf{L}^2_{r}(m_r(G))\right] \\
        \vdots \\
 m_{2N+1} (G) -  \mathsf{L}^{2N+1}_{0} (\rho(G)) - 
        \eps^2 \left[\mathsf{L}^{2N+1}_{1}(j(G)) +\displaystyle{ \sum_{r=2}^{2N+1}}  \mathsf{L}^{2N+1}_{r}(m_r(G))\right]
 \end{pmatrix} 
$$
can be decomposed into the sum of 
\begin{itemize}
\item the map
$$
\begin{pmatrix} \rho(G)\\   j(G) \\ m_2(G) \\ \vdots \\ m_{2N+1}(G) \end{pmatrix} \mapsto 
\begin{pmatrix} \rho(G)- K_{ \na_v \mu }( (- \Delta)^{-1} \rho(G)) \\ 
j(G)-  \mathsf{L}^1_{0} (\rho(G))    \\
 m_2 (G) -  \mathsf{L}^2_{0} (\rho(G))  \\
        \vdots \\
 m_{2N+1} (G) -  \mathsf{L}^{2N+1}_{0} (\rho(G)) 
 \end{pmatrix} 
$$
which is invertible, using the Penrose stability condition as in Lemma~\ref{final} and the underlying triangular structure,
\item and another bounded map, of size $\lesssim \eps^2$.
\end{itemize}
 It is therefore invertible for $\eps$ small enough.

We therefore deduce
\begin{lem}
For $t \lesssim  \min (T_\e, \lambda \delta^{-2/7})$, we have the bound
\begin{equation}
\| \rho(G) , j(G), (m_\ell(G))_{\ell \in \{2,\cdots, 2N+1\}} \|_{L^2(0,t; L_x^2)} \lesssim  \left\| \mathsf{R} \right\|_{L^2(0,t; L^2_x)} + \sum_{\ell = 1}^{2N+1}   \left\| \mathsf{R}^\ell \right\|_{L^2(0,t; L^2_x)}
 \end{equation}
in which  $\mathsf{R}$ and $ \mathsf{R}^\ell$ are defined as in \eqref{def-RRR}. 
\end{lem}

For what concerns the contribution of $\mathcal{S}_0$ and $\mathcal{S}_0^\ell$, an improvement of Lemma~\ref{lem-init} is required in order to reach arbitrary orders of time in $\eps$. We can prove Lemma~\ref{lem-init} holds with the sole constraint $t\leq \min(T_\eps, \lambda \delta^{-2/(2n+3)})$.

\begin{lem}
\label{lem-init2}
 For all $t\leq \min(T_\eps, \lambda \delta^{-2/(2n+3)})$, we have the bound
$$
 \left\| \mathcal{S}_0\right\|_{L^2(0,t;L^2_x)} +  \sum_{\ell=2}^{2N+1}\left\| \mathcal{S}_0^\ell\right\|_{L^2(0,t;L^2_x)} \lesssim \| g_{\vert_{t=0}} \|_{H^n_k},
$$
with $k>2N+11/2$.
\end{lem}

\begin{proof}[Proof of Lemma~\ref{lem-init2}] 

We focus only on $\mathcal{S}_0$ (the analysis being identical for $\mathcal{S}_0^\ell$). We follow the beginning of the proof of Lemma~\ref{lem-init} which remains unchanged. To get rid of the constraint in $\eps$, it is useful to apply the change of variables $p= \hat{v}$, which yields
\begin{align*}
&\int_{B(0, 1/\eps)} \pa^\alpha_x g_{|t=0} (x-t p, \Psi(0,t,x,v(p))) \\
&\qquad \times |\det \na_v \Phi(t,x,\Psi(0,t,x,v(p)))| | \det \na_v \Psi(0,t,v(p))| \frac{1}{(1-\eps^2|p|^2)^{5/2}}\, dp,
 \end{align*}
 with $v(p) = \frac{p}{\sqrt{1-\eps^2|p|^2}}$.
We can write the identity 
$$
\begin{aligned}
\pa^\alpha_x g_{|t=0} &(x-t p, \Psi(0,t,x,v(p))) \\
&= \pa_{x_k} \pa^{\alpha'}_x g_{|t=0} (x-t p, \Psi(0,t,x,v(p))) \\
&= - \frac{1}{t}\Bigg(\pa_{p_k} \big[\pa^{\alpha'}_x g_{|t=0} (x-t p, \Psi(0,t,x,v(p)))\big] \\
&\quad +  \pa_{p_k} \Psi(0,t,x,v(p)) \cdot \big[ \na_v \pa^{\alpha'}_x g_{|t=0}\big] (x-t p, \Psi(0,t,x,v(p))) \Bigg)  \\
&=: I'_1 +I'_2.
\end{aligned}
$$
We note that there is no analogue of the term $I_3$ that appears in the proof of Lemma~\ref{lem-init} and accounts for the limitation in terms of $\eps$.
We study the contributions of $I'_1$ and $I'_2$ exactly as we did as for that of $I_1$ and $I_2$ in the proof of Lemma~\ref{lem-init}, using bounds similar to those used in the proof of Proposition~\ref{propK}.
We conclude as before. 
\end{proof}

We treat the terms involving $-\eps \hat{v} \cdot \na_x (\A \cdot \na_v \mu) $ in $ \widetilde{\mathcal{S}_1}$ and $\widetilde{\mathcal{S}_1}^\ell$ as in Lemma~\ref{impro}.
Combining with the high order estimates of Lemma~\ref{newA} and the fact that the initial condition is well-prepared of order $p$ (so  that Lemma~\ref{newA-init} applies), this results in
\begin{lem}
For $t \lesssim  \min (T_\e, \lambda \delta^{-2/(2n+3)})$, there holds
$$
\begin{aligned}
\| \widetilde{\mathcal{S}_1} \|_{L^2(0,t; L^2_x)} &+ \sum_{\ell=2}^{2N+1} \| \widetilde{\mathcal{S}_1}^\ell \|_{L^2(0,t; L^2_x)} \\
&\lesssim  \eps^{p/2} \sqrt{t}  +  \eps^2\delta (1 + t^{n+5/2})M^2  + \eps^{2N+1}(1+t^{2})M^2.
 \end{aligned}
$$
\end{lem}


Therefore we obtain (recalling Lemma~\ref{av-j} to handle the average of $j(g)$),
\begin{lem}
For $t \lesssim  \min (T_\e, \lambda \delta^{-2/(2n+3)})$, we have
\begin{equation}
\label{newrj-final}
\begin{aligned}
\| \rho(g) , j(g), &(\widehat{m_\ell}(g))_{\ell \in \{2,\cdots, 2N+1\}} \ \|_{L^2(0,t; H^n_x)} \\
&\qquad  \lesssim   M_0+ \eps^{p/2} \sqrt{t} + \delta (1 + t^{n+3})M^2  + \eps^{2N+1}(1+t^{2})M^2,
\end{aligned}
\end{equation}
with $\widehat{m_\ell}(g) ={m_\ell}(g) -\langle {m_\ell}(g) \rangle$.
\end{lem}

Using~\eqref{newrj-final}, we can conclude the bootstrap argument as we did at the end of Section~\ref{Sec-5}, showing that 
$$
T_\eps \geq \lambda  \min (  \eps^{-\min(N+1/2,p)},  \delta^{-1/(n+3)}),
  $$
  for $\lambda>0$ small enough.
This concludes the proof of Proposition~\ref{key-prop3} and thus, arguing like in Section~\ref{Sec-6}, to that of Theorem~\ref{thm2}.


~\\

\noindent {\bf Acknowledgements.} 
DHK was partially supported by a PEPS-JCJC grant from CNRS.
TN's research was supported in part by the NSF under grant DMS-1405728. 
FR was partially supported  by the ANR-13-BS01-0003-01 DYFICOLTI.

\appendix 
\section{Scaling invariances}

We exploit here  some scaling invariances of the relativistic Vlasov-Maxwell system in order to express our result in the setting where the speed of light is considered as fixed.
The price to pay is that we consider data with special scaling.

\subsection{Slow variation in velocity}Let $\lambda>0$.
We note that $(\mathsf{f}, \mathsf{E}, \mathsf{B})$ is a solution to the relativistic Vlasov-Maxwell system with $c=1$ if and only if $(f^\lambda, E^\lambda,B^\lambda)$ defined as
\begin{equation}
\label{eq-scaling}
\begin{aligned}
f^\lambda(t,x,v) &:= \lambda^{-1} \mathsf{f} (\lambda t, x, v /\lambda), \\
E^\lambda(t,x,v) &:= \lambda^{2} \mathsf{E} (\lambda t, x), \\
B^\lambda(t,x,v) &:= \lambda^{2} \mathsf{B} (\lambda t, x)
\end{aligned}
\end{equation}
is a solution to the relativistic Vlasov-Maxwell system with the speed of light $c= 1/\lambda$.
It therefore follows that it is possible to translate the results of this paper into the following  setting: we consider that the speed of light is \emph{fixed}, set to $1$ for simplicity, and we consider special initial data $(\mathsf{f}_0, \mathsf{E}_0, \mathsf{B}_0)$ such that
\begin{equation}
\label{eq-data}
\begin{aligned}
\mathsf{f}_0 (x,v) &= \lambda \left( \mu(\lambda v) + \delta f_0(x,\lambda v) \right), \\
\mathsf{E}_0 (x) &= \delta \lambda^{-2} E_0 (x), \\
\mathsf{B}_0 (x)&= \delta \lambda^{-2} B_0 (x),
\end{aligned}
\end{equation}
with $\mu$ and $(f_0, E_0,B_0)$ satisfying the assumptions of Theorem~\ref{thm2} with $\eps= \lambda$. 
We deduce from Theorem~\ref{thm2} the following statement (which just consists in a translation in these new variables), which can be understood as a stability property of the homogeneous equilibrium $\lambda \mu(\lambda v)$ for small values of $\lambda$. Note that $\int_{\R^3} \lambda \mu(\lambda v) \, dv = 1/\lambda^2$, which means that such data are large to this extent.

  \begin{corollary}
  \label{coro}
Let $N\geq 2$, $p \geq 2$. Let $n\geq 6$ and $k>2N +11/2$.  
For initial data $(\mathsf{f}_0, \mathsf{E}_0, \mathsf{B}_0)$ satisfying the above requirements, there are $\lambda_0, \delta_0>0$ such that for all $\lambda\leq \lambda_0$ and $\delta \leq \delta_0$, the following holds. There is a unique smooth solution  $(\mathsf{f}, \mathsf{E}, \mathsf{B})$ of the Vlasov-Maxwell system ~\eqref{clVM0} with $c=1$, starting from $(\mathsf{f}_0, \mathsf{E}_0, \mathsf{B}_0)$, on the time interval $$I_{\lambda,\delta}: = [0,  c_0 \min ( \lambda^{-\min \{N-1/2,p-1\}},   \lambda \delta^{-1/(n+3)})].$$
   In addition, we have 
\begin{equation}
\label{result3}
\begin{aligned}
&\sup_{ [0,t]}  \| \mathsf{f}(s,x,v) - \lambda \mu(\lambda v) \|_{H^n_{x,v}}  \lesssim \delta \lambda^{-2} (1+t^{n+1/2}) M_0,
\\ &\| ( \mathsf{E},  \mathsf{B}) \|_{L^2(0,t;H^n_x)}  \lesssim \delta \lambda^{-2} (1+ \lambda^{p/2-1}\sqrt{t} + \lambda\delta (1 + t^{n+5/2})  + \lambda^{2N}(1+t^{2})   ) M_0,
\end{aligned}
 \end{equation} for $t\in I_{\lambda,\delta}$. 
  \end{corollary}
Here we may for instance choose $\delta = \lambda^\alpha$, with $\alpha>n+3$.

\subsection{Slow variation in space and time}Let $\lambda>0$.
As in \cite{BK}, we note that $(\mathsf{f}, \mathsf{E}, \mathsf{B})$ is a solution to the relativistic Vlasov-Maxwell system set on the torus 
$\mathbb{T}_{1/\lambda}^3:= \R^3 / (\frac{1}{\lambda} \Z^3)$,
with $c=1$ and if and only if $(f^\lambda, E^\lambda,B^\lambda)$ defined as
\begin{equation}
\label{eq-scaling-2}
\begin{aligned}
f^\lambda(t,x,v) &:= \lambda^{-3} \mathsf{f} ( t/\lambda^3, x/\lambda^2,  \lambda v ), \\
E^\lambda(t,x,v) &:= \lambda^{-4} \mathsf{E} ( t/\lambda^3, x/\lambda^2), \\
B^\lambda(t,x,v) &:= \lambda^{-4} \mathsf{B} ( t/\lambda^3, x/\lambda^2)
\end{aligned}
\end{equation}
is a solution to the relativistic Vlasov-Maxwell system (set on the torus $\T^3$) with the speed of light $c= 1/\lambda$.
As before it is possible to translate the results of this paper into the following physical setting: we consider that the speed of light is \emph{fixed}, set to $1$ for simplicity, and we consider special initial data $(\mathsf{f}_0, \mathsf{E}_0, \mathsf{B}_0)$ such that
\begin{equation}
\label{eq-data-2}
\begin{aligned}
\mathsf{f}_0 (x,v) &= \lambda^3 \left( \mu(v/\lambda) + \delta f_0(\lambda^2 x, v/\lambda) \right), \\
\mathsf{E}_0 (x) &= \delta \lambda^{4} E_0 (\lambda^2 x), \\
\mathsf{B}_0 (x)&= \delta \lambda^{4} B_0 (\lambda^2 x),
\end{aligned}
\end{equation}
with $\mu$ and $(f_0, E_0,B_0)$ satisfying the assumptions of Theorem~\ref{thm2} with $\eps= \lambda$. 
We deduce from Theorem~\ref{thm2} the following statement (which just consists in a translation in these new variables), which can be understood as a stability property of the homogeneous equilibrium $\lambda^3 \mu(v/\lambda)$ for small values of $\lambda$, with respect to perturbations on $\T^3_{1/\lambda} \times \R^3$. 

  \begin{corollary}
  \label{coro2}
Let $N\geq 2$, $p \geq 2$. Let $n\geq 6$ and $k>2N +11/2$.  
For initial data $(\mathsf{f}_0, \mathsf{E}_0, \mathsf{B}_0)$ satisfying the above requirements, there are $\lambda_0, \delta_0>0$ such that for all $\lambda\leq \lambda_0$ and $\delta \leq \delta_0$, the following holds. There is a unique smooth solution  $(\mathsf{f}, \mathsf{E}, \mathsf{B})$ of the Vlasov-Maxwell system~\eqref{clVM0} on $\T^3_{1/\lambda} \times \R^3$ with $c=1$, starting from $(\mathsf{f}_0, \mathsf{E}_0, \mathsf{B}_0)$, on the time interval $$I_{\lambda,\delta}: = [0,  c_0 \min ( \lambda^{-\min \{N+7/2,p+3\}},   \lambda^{-3} \delta^{-1/(n+3)})].$$
   In addition, we have 
\begin{equation}
\label{result3-2}
\begin{aligned}
&\sup_{ [0,t]}  \| \mathsf{f}(s,x,v) -  \mu(v/\lambda) \|_{H^n_{x,v}}  \lesssim \delta \lambda^{-n}  (1+t^{n+1/2}) M_0,
\\ &\| ( \mathsf{E},  \mathsf{B}) \|_{L^2(0,t;H^n_x)}  \lesssim \delta  (1+ \lambda^{p/2-1}\sqrt{t} + \lambda\delta (1 + t^{n+5/2})  + \lambda^{2N}(1+t^{2})   ) M_0,
\end{aligned}
 \end{equation} for $t\in I_{\lambda,\delta}$, with the Sobolev norms being taken on $\T^3_{1/\lambda} \times \R^3$.
  \end{corollary}

\section{Non-radial equilibria}

Without the radial assumption on the equilibrium $\mu(v)$, we mention that the 
results and proofs of this paper can be adapted (for other {\bf stable} equilibria). We require that $\mu(v)$ is an equilibrium for the relativistic Vlasov-Maxwell system, that is we ask in addition that 
\begin{equation}
\label{mu-eq}
\int_{\R^3} \hat{v} \mu \, dv =0.
\end{equation}
For instance select three even non-negative smooth and fastly decaying functions on $\R$, $\mu_1, \mu_2, \mu_3$ (with $\mu_i \neq \mu_j$ for some $i,j$), that are normalized so that $\int_{\R} \mu_i \, dv =1$ and take $\mu(v)= \prod_{i} \mu_{i}(v_{i})$. We ask for the assumption~\eqref{Pen} instead of the radial one.
Then $\mu$ is non-radial, stable, and satisfies all requirements, including~\eqref{mu-eq} for all $\eps\geq 0$.

\subsection{Theorem~\ref{thm} without the radial assumption}

Theorem~\ref{thm} still holds, as we can modify the proof as follows. 
When $\mu$ is not assumed to be radial, the algebraic relation~\eqref{j-g} between $j(f)$ and $j(g)$ does not hold anymore.
However we note that we have
\begin{equation}
\label{jg-new}
j(f) = j(g)  - \eps \widetilde \lambda A  + \eps^3 \varphi(A),
\end{equation}
setting 
\begin{align*}
\widetilde \lambda&:= \int_{\R^3} \frac{1}{(1+ \eps^2 |v|^2)^{3/2}} \mu(v) \, dv, \\
\varphi(A) &:= - \sum_{i=1}^3 A_i  \int_{\R^3} \frac{v_i}{(1+ \eps^2 |v|^2)^{3/2}} \mu(v) v \, dv.
\end{align*}
The main outcome concerns the analogue of Lemma~\ref{EB} in this context. 
\begin{lemma}
\label{EB-nonradial}
The conclusions of Lemma~\ref{EB} hold for $t \lesssim \min (T_\eps, \eps^{-2})$.
\end{lemma}
\begin{proof}[Proof of Lemma~\ref{EB-nonradial}] By the relation~\eqref{jg-new}, we have the wave equation
\begin{equation}\label{wawawa2} 
\varepsilon^2 \partial_t^2 A - \Delta A + \eps^2 \widetilde \lambda A  = \varepsilon \P  j(g) +  \eps^4 \P \varphi(A),\end{equation}
with $\P$ denoting the Leray projection.
We use similar estimates which results in the following analogue of the key bound~\eqref{dtA}:
\begin{equation*}
\begin{aligned}
\| (\eps A, \varepsilon \partial_t A, \nabla_x A)(t)\|_{L^\infty(0,t;H^n_x)} \lesssim &\sqrt{t} \| j(g)\|_{L^2(0,t; H^n_x)}  + \int_0^t \eps^2 \| \eps A(s) \|_{H^n_x} \, ds \\
&+ \|(\eps A, \eps \partial_t A, \na A)_{|t=0} \|_{H^n_x }.
\end{aligned}
\end{equation*}
using $\|\P \varphi(A)(s) \|_{H^n_x}  \lesssim \|  A(s) \|_{H^n_x}$. There is $\lambda_0>0 $ small enough so that for $t \leq \lambda_0 \eps^{-2}$, we can absorb the contribution of $ \int_0^t \eps^2 \| \eps A(s) \| \, ds$, which yields 
\begin{equation*}
\| (\eps A,\varepsilon \partial_t A, \nabla_x A)(t)\|_{L^\infty(0,t;H^n_x)} \lesssim \sqrt{t} \| j(g)\|_{L^2(0,t; H^n_x)}  + \|(\eps A, \eps \partial_t A, \na A)_{|t=0} \|_{H^n_x },
\end{equation*}
from which we conclude as for Lemma~\ref{EB}.
\end{proof}
The end of the argument also applies without further modification.

\subsection{Theorem~\ref{thm2} without the radial assumption}
Theorem~\ref{thm2}, that is the extension to arbitrary orders of $1/\eps$, does not hold anymore. 
One first obstruction comes from the requirement $t \lesssim 1/\eps^{2}$ in the new Lemma~\ref{EB-nonradial}.
Furthermore, we note for what concerns the Darwin approximation that the contribution of the main \emph{linear} terms cancels by symmetry when $\mu$ is radial symmetric (see Lemmas~\ref{newA1-dt} and~\ref{lem-Mk}).
As a result, this has to be taken into account without the radial assumption and we have to stop at first order in the Darwin approximation. However, we can see that this procedure still allows to improve the order with respect to $\eps$ to $1/\eps^2$.

\bibliographystyle{plain}
\bibliography{eMHD}

\end{document}